\documentclass[12pt]{amsart}

\usepackage{amstext,amsfonts,amssymb,amscd,amsbsy,amsmath,verbatim}
\usepackage[backrefs,lite]{amsrefs} 
\usepackage{ifthen, fullpage}
\usepackage{extarrows}
\usepackage{color,tikz,multirow}
\usepackage{amsthm}
\usepackage{latexsym}
\usepackage[all]{xy}
\usepackage{enumerate}

\newtheorem{lemma}{Lemma}[section]
\newtheorem{theorem}[lemma]{Theorem}

\newtheorem{prop}[lemma]{Proposition}
\newtheorem{cor}[lemma]{Corollary}

\newtheorem{claim*}{Claim}
\newtheorem{thm}[lemma]{Theorem}

\theoremstyle{definition}

\theoremstyle{remark}
\newtheorem{defn}[lemma]{Definition}
\newtheorem{example}[lemma]{Example}
\newtheorem{remark}[lemma]{Remark}

\newcommand{\Spec}{\operatorname{Spec}}

\newcommand{\Proj}{\operatorname{Proj}}

\newcommand{\Pic}{\operatorname{Pic}}

\newcommand{\NE}{\operatorname{NE}}

\newcommand{\Tor}{\operatorname{Tor}}
\newcommand{\Tot}{\operatorname{Tot}}

\newcommand{\Hom}{\operatorname{Hom}} 

\newcommand{\kk}{\Bbbk}

\newcommand{\rank}{\operatorname{rank}}

\newcommand{\codim}{\operatorname{codim}}
\newcommand{\PP}{\mathbb{P}}
\renewcommand{\AA}{\mathbb{A}}
\newcommand{\GG}{\mathbb{G}}
\newcommand{\HH}{\mathrm{H}}

\newcommand{\ZZ}{\mathbb{Z}}
\newcommand{\QQ}{\mathbb{Q}}
\newcommand{\UU}{\mathrm{U}}
\newcommand{\VV}{\mathrm{V}}
\newcommand{\WW}{\mathrm{W}}

\newcommand{\NN}{\mathbb{N}}

\newcommand{\cc}{c}
\newcommand{\dd}{\mathbf{d}}
\newcommand{\ee}{\mathbf{e}}

\newcommand{\cO}{\mathcal{O}}
\newcommand{\cE}{\mathcal{E}}
\newcommand{\cF}{\mathcal{F}}
\newcommand{\cU}{\mathcal{U}}

\newcommand{\bK}{\mathbf{K}}

\newcommand{\FF}{\mathbf{F}}
\newcommand{\Gbull}{\mathbf{G}}

\newcommand{\GL}{{GL}}

\newcommand{\defi}[1]{\textsf{#1}} 

\newcommand{\zp}{\circ}
\newcommand{\nothing}{\emptyset}

\newcommand{\DD}{\mathrm{D}}

\newcommand{\CQ}{\mathrm{C}}
\newcommand{\CvbQ}{\mathrm{C}_{\text{vb}}}
\newcommand{\BBQ}{\mathrm{B}}
\newcommand{\BBirr}{{\mathrm{B}}^{\text{irr}}}

\renewcommand{\P}{{\mathbb P}}

\def\BS{Boij--S\"oderberg~}

\title{Categorified duality in Boij--S\"oderberg Theory and invariants of free complexes}

\author{David Eisenbud}
\address{Department of Mathematics, University of California, Berkeley, CA 94720, USA}
\email{de@math.berkeley.edu}

\author{Daniel Erman}
\address{Department of Mathematics, University of Wisconsin, Madison, WI 53706, USA}
\email{derman@math.wisc.edu}
\urladdr{http://www.math.wisc.edu/~derman/}
\thanks{The first author was partially supported by an NSF grant, and the second author was partially supported by a Simons Foundation fellowship and by NSF grant DMS-1302057.}

\begin{document}

\begin{abstract} 
We present a robust categorical foundation for the duality theory introduced by Eisenbud and Schreyer to prove the \BS conjectures describing numerical invariants of syzygies.
The new foundation allows us to extend the reach of the theory substantially.

More explicitly, we construct a pairing between derived categories that
simultaneously categorifies all the functionals used by Eisenbud and Schreyer.
With this new tool, we describe the cone of Betti tables of finite, minimal free complexes  having homology modules of specified dimensions over a polynomial ring, and we treat many examples beyond polynomial rings.  We also construct an analogue of our pairing between derived categories on a toric variety, yielding
toric/multigraded analogues of the Eisenbud--Schreyer functionals.

\end{abstract}

\maketitle

\vspace{-1cm}
\tableofcontents


\vspace{-1cm}
\section*{Introduction}
The Hilbert polynomial is a fundamental invariant of graded modules or coherent sheaves on projective space. This invariant is refined in two different ways by the Betti table of a graded module and the cohomology table of a coherent sheaf. Work of Eisenbud and Schreyer \cite{eis-schrey1} suggested a duality between these refinements that involves an infinite collection of bilinear pairings. In this paper we clarify the duality, showing that it is embodied in
a single pairing between derived categories.

Eisenbud and Schreyer's original goal was to prove the \BS conjectures~\cite{boij-sod1}
which describe the possible values of Betti tables of finite free resolutions of modules of finite length over a polynomial ring, up to scalar multiple.  
Now known as \BS theory, these results~\cite{efw,eis-schrey1} were subsequently extended to cover all free resolutions over a polynomial ring~\cite{boij-sod2,eis-schrey2},
special cases of resolutions over other rings~\cite{bbeg,beks-local} and over multigraded rings~\cite{boij-floystad,floystad-multigraded}. The theory has developed in other directions as well~\cite{beks-tensor, ees-filtering, erman-semigroup,nagel-sturgeon,sam-weyman}. All these developments rely on the foundations established by Eisenbud and Schreyer. 

Our categorification of the duality gives a new foundation for all of these developments and allows us to substantially extend the reach of the theory.  In particular we are able to characterize the cone of Betti tables of complexes with homology of a given codimension.
We also treat many rings other than polynomial rings, and we present a framework for an extension to the multigraded case of toric varieties.

In Part I of this paper, 
we construct the pairing, which takes values in a derived category of graded modules over a polynomial ring in 1 variable.

In Part II, we use the pairing to extend from the consideration of free resolutions to the consideration of more general complexes.
We treat the Betti numbers of finite free complexes with prescribed homology  (\S\ref{sec:refined}--\ref{sec:monads}) and clarify the Eisenbud--Schreyer duality results (\S\ref{sec:functionals}--\ref{sec:duality}).  

In Part III, we extend the theory to a wider class of graded rings (\S\ref{sec:functor}), and we discuss some applications to the study of infinite resolutions (\S\ref{sec:infinite}).  Lastly, we explain a natural generalization to the multigraded case (\S\ref{sec:toric}).

\addtocontents{toc}{\protect\setcounter{tocdepth}{-1}}
\subsection*{Categorifying the Eisenbud--Schreyer Duality}
\addtocontents{toc}{\protect\setcounter{tocdepth}{1}}
Let $\kk$ be a field. Let $A= \kk[t]$ and $S=\kk[x_0, \dots, x_n]$ be the polynomial rings in 1 and $n+1$ variables. If 
$$
\FF= [\cdots \gets \FF_i \gets \FF_{i+1}\gets \cdots ]
$$
is a bounded complex of finitely generated graded free $S$-modules, then $\beta_{i,j}\FF$ is defined to be the dimension of the degree $j$ component of the graded vector space $\Tor_i(\FF,\kk)$.  The \defi{Betti table} of $\FF$ is the vector with coordinates $\beta_{i,j}\FF$ in the vector space $\VV = \oplus_{i\in \ZZ} \oplus_{j\in \ZZ}\QQ$. Similarly, the \defi{cohomology table} of a bounded complex of coherent sheaves $\cE$ on $\PP^{n}$ is the vector with coordinates $\gamma_{i,j}\cE := h^{i}\cE(j)$ in the vector space $\WW = \oplus_{i\in \ZZ}\prod_{j\in \ZZ}\QQ$, where $h^{i}\cE(j)$ denotes the dimension of the $i$-th hypercohomology of the complex $\cE(j) := \cE \otimes \cO_{\PP^{n}}(j)$. 

Let $\DD^{b}(\P^{n}), \DD^{b}(S), \DD^{b}(A)$ denote the bounded derived categories of the categories of coherent sheaves on $\P^{n}$ and of finitely generated graded modules over $S$ and $A$, respectively.  Given a complex $\FF\in \DD^b(S)$, we write $\widetilde \FF$ for the corresponding complex of sheaves on $\PP^{n}$. 

The central construction of this paper is a functor
$$
\xymatrix{\DD^{b}(S)\times \DD^{b}(\PP^{n})  \ar[r]^-{\Phi}&\DD^{b}(A),}
$$
with the following properties:
\begin{theorem}\label{thm:Phi} If $\FF$ is a bounded complex of free graded $S$-modules and $\cE$ is a bounded complex of coherent sheaves on $\P^{n}$ then:
\begin{enumerate}
	\item\label{thm:Phi:1}  The Betti table of $\Phi(\FF,\cE)$ depends only on the Betti table of $\FF$ and the cohomology table of $\cE$.
	\item\label{thm:Phi:2}  If $\widetilde{\FF}\otimes \cE$ is exact, then $\Phi(\FF,\cE)$ is generically exact.  
\end{enumerate}
\end{theorem}
We will deduce from this theorem that, for any $1\leq k\leq n+1$, there is a pairing
 \begin{equation}
\label{eqn:**}
\begin{tabular}{c c c c c}
$\left\{\begin{matrix}
\text{Betti tables of} \\ \text{free $S$-complexes}\\
\text{ with homology of }\\ \text{ codimension $\geq k$}
\end{matrix}\right\}$
&
$\times$
&
$\left\{\begin{matrix}
\text{cohomology tables}\\
\text{of sheaves on } \PP^n \text { of }\\
\text{ codim. $\geq n+1-k$}
\end{matrix}\right\}$
&
$\longrightarrow$
&
$\left\{\begin{matrix}
\text{Betti tables of} \\
\text{generically exact}\\
 \text{free $A$-complexes}\\
\end{matrix}\right\},$
\end{tabular}
\end{equation}
where $(\beta(\FF),\gamma(\cE))\mapsto \beta(\Phi(\FF,\cE))$.  Extending this map linearly yields a bilinear pairing among the three positive, rational cones spanned by the tables from \eqref{eqn:**}, as  in Figure~\ref{fig:bracket}. 

We will show that this pairing gives a duality.
\begin{thm}\label{thm:duality}
Fix a point $v\in \VV$.   The following are equivalent:
\begin{enumerate}
	\item   $v$ is a positive, rational multiple of a Betti table $\beta(\FF)$ for some free complex $\FF$ where the homology of $\FF$ has codimension $k$;
	\item  Given any sheaf $\cE$ of codimension $n+1-k$, $v$ pairs with $\gamma(\cE)$ to give an element of the cone of Betti tables of generically exact free $A$-complexes.
\end{enumerate}
\end{thm}

The cone of Betti tables of generically exact free $A$-complexes is easy to describe, and it is easy to write down the nonnegative functionals that define it.
By composing the functor $\Phi$ with a nonnegative functional on this cone, we get all the functionals $\langle -,-\rangle_{\tau,\kappa}$ used by Eisenbud and Schreyer. In this sense, $\Phi(\FF,\cE)$ categorifies all of the nonnegative numbers $\langle \FF, \cE\rangle_{\tau,\kappa}$ used in \cite{eis-schrey1}. Allowing complexes with homology is essential, as the functor $\Phi$ does not respect the property of being a resolution: even if $\FF$ is a resolution of a finite length module the complex $\Phi(\FF,\cE)$ may fail to be a resolution. Moreover, the results on resolutions follow easily from these more general results on complexes.

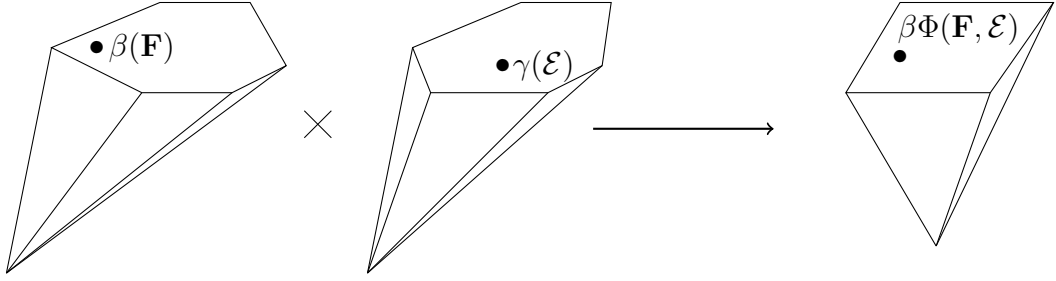
\begin{figure}
\begin{tikzpicture}[scale=1.2]
\draw[-](.5,.5)--(1,3);
\draw[-](.5,.5)--(2,2.5);
\draw[-](.5,.5)--(3,2.5);
\draw[-](.5,.5)--(3.6,2.8);
\draw(1.5,3.0) node {$\bullet$};
\draw(2.0,3.0) node {$\beta(\FF)$};
\draw[-](1,3)--(2,2.5)--(3,2.5)--(3.6,2.8)--(3.2,3.5)--(2.2,3.5)--cycle;
\draw[-] (3.8,2)--(4.1,2.3);
\draw[-] (4.1,2)--(3.8,2.3);
\draw[-](4.5,.5)--(5,3);
\draw[-](4.5,.5)--(5.2,2.5);
\draw(6.0,2.8) node {$\bullet$};
\draw(6.45,2.8) node {$\gamma(\cE)$};
\draw[-](4.5,.5)--(6.5,2.5);
\draw[-](4.5,.5)--(7.1,2.8);
\draw[-](5,3)--(5.2,2.5)--(6.5,2.5)--(7.1,2.8)--(7.2,3.5)--(6.2,3.5)--cycle;
\draw[->,thick](7,2.1)--(9,2.1);
\draw[-](10.8,.8)--(9.8,2.5);
\draw[-](10.8,.8)--(11.4,2.5);
\draw[-](10.8,.8)--(12.1,3.5);
\draw(10.4,2.9) node {$\bullet$};
\draw(11.05,3.2) node {$\beta\Phi(\FF,\cE)$};
\draw[-](9.8,2.5)--(11.4,2.5)--(12.1,3.5)--(10.4,3.5)--cycle;
\end{tikzpicture}
\caption{The duality between Betti tables and cohomology tables involves three cones.  Namely, given the Betti table $\beta(\FF)$ of a complex of $S$-modules, and the cohomology table $\gamma(\cE)$ of a complex of coherent sheaves on $\PP^n$, we use our pairing to produce the Betti table $\beta(\Phi(\FF,\cE))$ of a complex of $A$-modules. 
}
\label{fig:bracket}
\end{figure}

%
\addtocontents{toc}{\protect\setcounter{tocdepth}{-1}}
\subsection*{Decomposing the Betti tables of complexes}
\addtocontents{toc}{\protect\setcounter{tocdepth}{1}}


One of the most important consequences of the duality theory of~\cites{eis-schrey1,boij-sod2} is that the Betti table of any minimal free resolution over $S$ decomposes in a nice way as a positive rational linear combination of pure resolutions: that is, resolutions of Cohen-Macaulay modules with the property that for each $i$ at most one $\beta_{i,j}$ is nonzero. (See~\cites{eis-schrey-icm,floystad-expository} for expository introductions to \BS theory.)

A consequence of the categorified theory is that these same building blocks suffice for the decomposition of the Betti table of much more general complexes.
Theorem~\ref{thm:extremal rays refined} provides the full statement and proof, and implies the previously known results.

%

 Here is the special case of our result when the homology has finite length. By a (homologically) shifted resolution, we mean a complex of
finitely generated graded free modules the form
$$
\FF= [\FF_{k}\leftarrow\cdots \FF_{k+\ell}\leftarrow 0]
$$
that has homology only at $\FF_{k}$.
 
\begin{cor}\label{cor:decompose}
Let $\FF\in \DD^b(S)$ have finite length homology.  Then $\beta(\FF)$ is a positive rational combination of Betti tables of shifted free resolutions of modules of finite length. 
\end{cor}
As in the case of resolutions, the decomposition is algorithmic and, in a certain sense, unique.

We illustrate Corollary~\ref{cor:decompose} with an example.  By convention we display the Betti table of $\FF$ as a table
of integers where the element of the $i$-th column and $j$-th row is $\beta_{i,i+j}\FF$, and we replace each zero with $-$. For clarity we often decorate the $(0,0)$ entry
with a superscript $\circ$. In displays of complexes we often suppress the terms that are zero.

\begin{example}
Let $S=\kk[x,y]$ and consider the complex:
\[
\FF := \left[S^1\overset{\left(\begin{smallmatrix}x&y\end{smallmatrix}\right)}{\xlongleftarrow{\hspace*{1.1cm}}} S^2(-1)\overset{\left(\begin{smallmatrix}-y^2&xy\\xy&-x^2\end{smallmatrix}\right)}{\xlongleftarrow{\hspace*{1.1cm}}} S^2(-3)\overset{\left(\begin{smallmatrix}y\\x\end{smallmatrix}\right)}{\xlongleftarrow{\hspace*{1.1cm}}} S^1(-4)\right],
\]
which has has Betti table 
$$
\beta(\FF)=\begin{bmatrix} 1^\circ&2&-&-\\-&-&2&1\end{bmatrix}.
$$
The complex $\FF$ has finite length homology $H_{0}\FF = \kk,\ H_{1}\FF = \kk(-2)$. 

To decompose the Betti table of $\FF$, we consider the
modules $M:=S(-1)/(x^2,xy,y^2)$, and $N:=\Hom(M(1),\kk)$.  The Betti tables of (the minimal free resolutions of) $M[1]$ and $N$ are
\[
\beta(M[1])=\begin{bmatrix}
-^\zp&1&-&-\\
-&-&3&2
\end{bmatrix}
\quad \text{ and } \quad
\beta(N)=\begin{bmatrix}
2&3&-&-\\
-&-&1&-
\end{bmatrix},
\]
and we thus have the decomposition:
\[
\beta(\FF)=
\frac{1}{2}\beta(M[1])
+
\frac{1}{2}\beta(N).
\]

We note that $\beta(\FF)$ can \emph{not} be written as a positive \emph{integral} combination of Betti tables of minimal resolutions of modules of finite length: First, $\beta(\FF)$ is not itself a resolution, since it has length 3 and $S$ has dimension only 2. 
Further,  the sum of the Betti numbers of any resolution of a nonzero $S$-module of finite length is at least 4, while the sum of the Betti numbers of $\FF$ is $6<2\cdot 4$.
\end{example}

\subsection*{Beyond Polynomial Rings}
Our description of the cone of Betti tables of bounded complexes
extends to a wide class of rings in the following way. 
Let $S\subseteq R$ be a finite extension of graded rings, and let
$X=\Proj(R)$.  We let $f\colon X\to \PP^n$ denote the corresponding finite
map of projective schemes of dimension $n$, and we set $L:=f^*\cO(1)$.
%

%

We say that $\cU$ is an \defi{Ulrich sheaf} for $f$ if $f_*(\cU)\cong \cO_{\PP^n}^r$ for some $r>0$.  It was pointed out in \cite[Theorem~5]{eis-schrey-abel} that the existence of an Ulrich sheaf for $f$ implies that  the cone of cohomology tables of vector bundles on $X$ is the same as that on $\PP^{n}$. (The theorem is stated there when $L$ is very ample, but the proof
carries over to this more general situation.) 

The situation for Betti tables of resolutions of finite length modules over $R$ is not at all analogous to the situation over $S$.  In fact, previous results show that the structure is somewhat complicated even in the case of a graded hypersurface ring of low embedding dimension~\cite{bbeg}, or in the case of  a Veronese subring of $\kk[x_1,x_2]$ as in~\cite{kummini-sam}. But we prove a complete analogy for bounded complexes with finite length homology:

\begin{cor}\label{cor:isom cones}
Let $R$ be a graded $S$-algebra such that the map $f\colon \Proj(R)\to \P^{n}$ is finite.  If $\Proj(R)$ admits an Ulrich sheaf for $f$, then the cone of Betti tables
of bounded free complexes with finite length homology over  $R$ is the same
as the cone of Betti tables of bounded free complexes with finite length homology over $S$. 
\end{cor}
This provides the first description of a cone of Betti tables over $R$ for many new rings $R$.  These include, for instance, the first cases where $R$ is not generated in degree $1$ (see Examples~\ref{ex:elliptic} and \ref{ex:K3}) and the first cases where $R$ fails to be Cohen-Macaulay (see Example~\ref{ex:curves}).

One direction of the proof is easy, as the pullback of an $S$-complex with finite length homology is an $R$-complex with finite length homology.  For the other direction, we apply the duality statement of Theorem~\ref{thm:duality} to place limits on Betti tables over $R$. 


\begin{example}\label{ex:elliptic}
Let $E$ be a genus one curve and let $L=2P$, where $P$ is any (degree one) point of $E$.  The map $f$ corresponding to the complete
linear series $|L|$ maps $E$ two-to-one to $\P^{1}$. The ring $R(E,L)=\oplus_{e\in \NN} H^0(E,\cO_E(eL))$ has the form
$\kk[x_1,x_2,y]/(g(x_{1},x_{2},y))$  where $\deg(x_i)=1$, $\deg(y)=2$, and where $\deg(g)=4$.
If $P\neq Q\in E$ then the sheaf $\cO_E(3P-Q)$ is an Ulrich sheaf for $f$. 
Corollary~\ref{cor:isom cones} thus implies that the cone of
Betti tables of bounded free complexes with finite length homology over $R(E,L)$ is the same
as the corresponding cone over $\kk[x_1,x_2]$.  
\end{example}


For all of the new graded rings $R$ covered by Corollary~\ref{cor:isom cones}, there exist finitely generated $R$-modules of infinite projective dimension.  It would thus be natural to also consider bounded below elements of the derived category of graded $R$-modules.  

In \S\ref{sec:infinite}, we take a different approach: by realizing an infinite resolution as a limit of bounded complexes, we apply Corollary~\ref{cor:isom cones} prove a decomposition theorem for the Betti table of an infinite resolution of a finite length module.
\begin{example}
Let $R=\kk[x,y,z]/(x^2,xy)$ and let $\FF$ be the minimal free resolution of $R/(x,y,z^2)$.  Then the Betti table of $\FF$ decomposes as an positive, rational, infinite sum of shifted free resolutions of modules of finite length:
\begin{align*}
\beta(\FF)&=\begin{bmatrix}1^\zp&2&3&5&8&\dots \\ -&1&2&3&5&\dots \end{bmatrix}\\
&=\begin{bmatrix}\frac{1}{3}^\zp&-&-&\dots \\ -&1&\frac{2}{3}&\dots \end{bmatrix}
+
\begin{bmatrix}\frac{2}{3}^\zp&1&-&\dots \\ -&-&\frac{1}{3}&\dots \end{bmatrix}
+
\begin{bmatrix}-^\zp &\frac{1}{3}&-&-&\dots \\ -&-&1&\frac{2}{3}&\dots \end{bmatrix}\\
&\ +
\begin{bmatrix}-^\zp&\frac{2}{3}&1&-&\dots \\ -&-&-&\frac{1}{3}&\dots \end{bmatrix}
+
\begin{bmatrix}-^\zp&- &\frac{2}{3}&-&-&\dots \\ -&-&-&2&\frac{4}{3}&\dots \end{bmatrix}
+
\cdots.
\end{align*}
\end{example}

\subsection*{The Multigraded Case}
Our construction of $\Phi$ naturally generalizes to the multigraded case.  This offers a new perspective on the potential for extending \BS theory to toric varieties.  Let $X$ be a projective toric variety and let $R$ be the Cox ring of $X$ with the natural $\Pic(X)$ grading and its natural irrelevant ideal. We say that a complex $\FF$ over $R$ has \defi{irrelevant homology} if its homology is supported on the irrelevant ideal.

In place of the ring $A=\kk[t]$, we take the semigroup ring $C=\kk[\NE(X)]$, where $\NE(X)\subseteq \Pic(X)$ is the subsemigroup of effective divisors.  
The ring $C$ is graded by the group $\Pic(X)$.

Let $\DD^b(R)$ and $\DD^b(C)$ denote the bounded derived categories of finitely generated, multigraded $R$-modules and $C$-modules, respectively.   
If $\FF\in \DD^b(R)$ then $\widetilde{\FF}$  denotes the corresponding complex of coherent sheaves in $\DD^b(X)$.

Generalizing the construction above, we construct a functor
\[
\Phi_{X}: \DD^b(R)\times \DD^b(X)\to \DD^b(C)
\]
with the following properties:
\begin{theorem}\label{thm:Phimulti}
If $\FF$ is a bounded complex of free multigraded $R$-modules and $\cE$ is a bounded complex of coherent sheaves on $X$, then:
\begin{enumerate} 
	\item\label{thm:Phi':1}  The multigraded Betti table of $\Phi_{X}(\FF,\cE)$ depends only on the multigraded Betti table of $\FF$ and the multigraded cohomology table of $\cE$.
	\item\label{thm:Phi':2}  If $\widetilde{\FF}\otimes \cE$ is exact, then $\Phi_{X}(\FF,\cE)$ is generically exact.  
\end{enumerate}
\end{theorem}

For example, if $\FF$ has irrelevant homology, then we obtain a pairing of the form:
\begin{equation*}
\label{eqn:multipairing}
\left\{\begin{matrix}
\text{Multigraded Betti tables} \\ \text{of free $R$-complexes}\\
\text{with irrelevant homology}\end{matrix}\right\}
\times 
\left\{\begin{matrix}
\text{cohomology }\\
\text{tables of vector}\\
\text{ bundles on } X
\end{matrix}\right\}
\longrightarrow
\left\{\begin{matrix}
\text{Multigraded Betti tables} \\ \text{of generically exact }
\\ \text{ free $C$-complexes}
\end{matrix}\right\}
\end{equation*}
This pairing enables the construction of toric/multigraded analogues of the Eisenbud--Schreyer functionals.  It also suggests that, at least from the duality viewpoint, the cone of free $R$-complexes with irrelevant homology may be a more natural object to study than the cone of free resolutions over $R$.

\addtocontents{toc}{\protect\setcounter{tocdepth}{-1}}
\section*{Acknowledgments}
\addtocontents{toc}{\protect\setcounter{tocdepth}{1}}
We thank Frank-Olaf Schreyer for many insights and inspirations.
We also thank Christine Berkesch, Bhargav Bhatt, A.J.~de Jong, Rob Lazarsfeld, Ezra Miller, W.~Frank Moore, Steven Sam, Gregory G.~ Smith, and Ravi Vakil for conversations and suggestions that improved this paper.  We thank Stepan Paul for suggested corrections in Section~\ref{sec:toric}.
Some of this work was completed during the first author's visit to the University of Michigan, and we are grateful for their hospitality.  The computer algebra system \texttt{Macaulay2} \cite{M2} provided valuable assistance throughout our work.  Finally, we thank a referee of an earlier draft for detailed suggestions that significantly improved the paper.

\section*{\underline{{Part I: Categorifying the Duality in \BS Theory}}}
\section{Notation}\label{sec:notation}
We gather some notation and definitions that we will use throughout.  We denote by  $\mathfrak m=(x_0, \dots, x_n)$ the homogeneous maximal ideal on $S$.
\begin{defn}
If $\FF\in \DD^b(S)$ is a free complex, then we say that $\FF$ is \defi{minimal} if each differential $\partial: \FF_i\to \FF_{i-1}$ satisfies $\partial(\FF_i)\subseteq \mathfrak m\FF_{i-1}$.
\end{defn}
We may represent any $\FF\in \DD^b(S)$ by a minimal, free complex.  Under this assumption, we may write $\FF_i$ as the direct sum $\FF_i=\oplus_{j\in \ZZ} S(-j)^{\beta_{i,j}\FF}$.  If $\FF$ is quasi-isomorphic to a complex with only one nonzero term, then we say that $\FF$ is a \defi{shifted resolution}.  We denote the $r$th homology module of $\FF$ by $\HH_r\FF$.

\begin{defn}\label{def:deg seq}
A \defi{ degree sequence of codimension $\ell$} is a strictly increasing indexed sequence of $\ell+1$ 
integers, padded by  infinite strings of $-\infty$ on the left and of $\infty$ on the right. More formally,
it is a sequence of the form
\[{\dd}=(\dots, d_i, d_{i+1}, \dots)
\]
with  $d_{i} \in \{-\infty\}\cup \ZZ\cup \{\infty\}$ and $d_i \leq d_{i+1}-1$ and 
where  $\ell+1$ entries of $\dd$ lie in $\ZZ$. 
We partially order the degree sequences  termwise: $d\leq d'$ if and only if $d_i\leq d_i'$ for all $i$.
\end{defn}
(This usage is slightly more general than that of \cite{eis-schrey1} or \cite{boij-sod2}, where degree
sequences were taken to be what would be written here as
$(\dots,-\infty,d_{0}^{\zp},\dots,d_{\ell}, \infty,\dots)$.)

As with Betti tables, we use a $\zp$ to indicate homological position zero when writing a degree sequence. 
Thus for example
$$
(\dots, -\infty , 0, 1^{\circ}, 3, \infty, \dots) < (\dots, -\infty , 0, 1, 3^{\circ}, \infty, \dots) 
$$
are  degree sequences of codimension 2.

Given any degree sequence $\dd$, we say that a complex $\FF$ is \defi{pure of type $\dd$} if: for all $i$ such that $d_i\in \ZZ$, the free module $\FF_i$ is generated entirely in degree $d_i$; and if $\FF_i=0$ when $d_i=\pm \infty$.   The existence of pure resolutions (see~\cite{efw} or \cite[\S5]{eis-schrey1}) shows that, for any  degree sequence $\dd$ of codimension $\ell\leq n+1$, there exists a shifted resolution of a Cohen-Macaulay module of codimension $\ell$ over $S$ that is a pure complex of type $\dd$.  

If $P$ is some property of graded $S$-modules, we say that a complex $\FF \in \DD^b(S)$ has property $P$ if the direct sum of the homology modules of $\FF$ has property $P$. We extend the definition of  properties of coherent sheaves to $\DD^b(\PP^n)$ similarly.  

A \defi{root sequence $f$ of dimension $s$} is a strictly decreasing sequence
of $s$ integers, $f=(f_1>\cdots>f_s)$.  
A sheaf $\cE$ on $\PP^{n}$ is
\defi{supernatural of type} $f=(f_1, \dots, f_{s})$ if the following are satisfied: 
\begin{enumerate}
\item The dimension of $\cE$ is $s$.
\item For all $j\in \mathbb Z$, there exists at most one $i$ 
		such that $\dim_\Bbbk H^i(\PP^{n}, \cE(j))\ne 0$.
\item The Hilbert polynomial of $\cE$ has roots $f_1, \dots, f_{s}$.
\end{enumerate}
For every root sequence $f$ of dimension $s\leq n$, there exists a supernatural sheaf of type
$f$~\cite[Theorem~0.4]{eis-schrey1}.
Moreover, the cohomology table of any coherent sheaf 
can be written as a positive real combination of cohomology tables 
of supernatural sheaves~\cite[Theorem~0.1]{eis-schrey2}.  


We index complexes in $\DD^b(S)$ homologically as in $\FF=[\dots \gets \FF_0\gets \FF_1\gets \dots]$.  For any $k\in \ZZ$, we define a \defi{shift} of $\FF$, denoted $\FF[k]$, as the complex obtained by shifting the indices in the following way:  $(\FF[k])_i=\FF_{i-k}$.

Notation for various cones of Betti tables is introduced in \S\ref{sec:refined}.

\section{The categorical duality}\label{sec:duality pairing}
In this section we define the functor $\Phi$ and derive strengthened versions of its properties given in Theorem~\ref{thm:Phi}. 
 Let 
$\sigma\colon S\to S\otimes A = S[t]$
be the homomorphism defined by $\sigma(x_{i})=x_{i}t$. 
We write $-\otimes_\sigma S[t]$ to denote tensoring over $S$ with $S[t]$ using the structure
given by $\sigma$. Note that $\sigma$ is not a flat map---it is not even equidimensional.

If $F$ is a graded  $S$-module, then 
$$
F\otimes_{\sigma} S[t]
$$
is a bigraded $S[t]$ module that defines a graded sheaf 
$$
\tau(F) := \widetilde{F\otimes_{\sigma} S[t]}
$$
on $\PP^{n}_{A} = \PP^{n}\times \AA^{1}$, where the grading comes from the grading of $A = \kk[t]$.
This functor extends to a functor $\tau$ on derived
categories
taking a graded complex of free $S$-modules $\FF$ to
$$
\tau(\FF): =\widetilde \FF \otimes_{\sigma}\cO_{\PP^{n}\times \AA^{1}},
$$
a complex of graded sheaves on $\PP^{n}\times \AA^{1}$.  The main practical effect of this
is that in a complex of graded, free $S$-modules, the forms $f_{ij}$ involved in a matrix
representing a differential are replace by $t^{\deg(f_{ij})}f_{ij}$.
This description of $\tau$ could 
be extended to graded complexes of arbitrary finitely generated graded modules
at the expense of replacing the tensor product with a derived tensor product, but we
will never need this.

We consider an example.  If
$$
\FF= [\xymatrix{0&S\ar[l]&S(-e)\ar[l]_f& 0\ar[l]}]
$$
where $f$ is a form of degree $e$, then
$$
\tau(\FF)= [\xymatrix{0&\cO_{\PP^{n}}\boxtimes A\ar[l]&\cO_{\PP^{n}}(-e)\boxtimes A(-e)\ar[l]_-{t^ef}& 0\ar[l]}]
$$
where $P\boxtimes Q$ denotes the tensor product of the pullbacks of $P$ and $Q$ from
$\PP^{n}$ and $\AA^{1}$, respectively. 


\begin{defn} \label{defn:product} The functor $\Phi: \DD^{b}(S)\times \DD^b(\PP^n) \to \DD^{b}(A)$ is the composition of $\tau$ with the projection $Rp_{2*}$ to the derived category of graded $A$-modules; that is,
$$
\Phi(\FF,\cE) = Rp_{2*} \bigl(\tau(\FF)\otimes_{\P^{n}\times\AA^{1}} (\cE\boxtimes \cO_{\AA^{1}}) \bigr)
$$
where $\FF$ denotes a graded  complex of finitely generated free $S$-modules. We  often write
$\FF\cdot \cE$ for $\Phi(\FF,\cE)$.
\end{defn}
We will only need to use the definition of $\Phi$ in the special case where $\FF$ is a free complex and where $\cE$ is (the extension by zero of) a vector bundle on a linear subspace of $\PP^n$.

For those comfortable with stacks, Definition~\ref{defn:product}
could be rephrased as follows. Consider the commutative diagram:
\[
\xymatrix{
\PP^n&\PP^{n}\times [\AA^1/\GG_m]\ar[l]_-{\pi_{1}} \ar[r]^-{\Sigma} \ar[d]^{\pi_2}&[\AA^{n+1}/\GG_m]\\
&[\AA^1/\GG_m]&
}
\]
where $\Sigma$ is the morphism induced by $\sigma$ and the maps $\pi_1$ and $\pi_2$ are the projections.  We could define $\Phi(\FF,\cE)$ to be $R\pi_{2*}\left( \Sigma^*\FF\otimes \pi_{1}^{*}\cE\right)\in \DD^b([\AA^1/\GG_m])$.

To see why this is an equivalent definition, note first  that there is an equivalence of categories (given by pullback/descent) between coherent sheaves on $[\AA^1/\GG_m]$ and graded, finitely generated $A$-modules. Further, since the covering map $\AA^1\to [\AA^1/\GG_m]$ is flat, cohomology commutes with base change (see \cite[0765]{stacks-project}) for the diagram
\[
\xymatrix{
\PP^n\times \AA^1\ar[r]\ar[d]^-{p_2}&\PP^{n}\times [\AA^1/\GG_m]\ar[d]^{\pi_2}\\
\AA^1\ar[r]&[\AA^1/\GG_m].
}
\]
Thus, the pullback of $R\pi_{2*}\left( \Sigma^*\FF\otimes \pi_{1}^{*}\cE\right)$ is quasi-isomorphic 
to $\Phi(\FF,\cE)$.

We remark that the functor $\Phi$ specializes to certain Fourier--Mukai transforms. If we fix a complex $\FF$, then using the notation of \cite[Definition~3.3]{caldararu}, the functor $\Phi(\FF, -)$ is the Fourier--Mukai $\Phi^{\tau(\FF)}\colon \DD^b(\PP^n)\to \DD^b(\AA^1)$.

Here is a sample computation of $\Phi$: 

\begin{example} Let 
$$
\bK = \bigl[ S\gets S^{n+1}(-1) \gets \wedge^{2}(S^{n+1})(-2) \gets\cdots\gets S(-n-1)\bigr]
$$
be the Koszul complex, the minimal free resolution of $\kk$, and take
$\cE = \cO_{\PP^{n}}$, so that 
$$
\bK \cdot \cE = Rp_{2*}\tau(\bK).$$  
There is a spectral sequence
\[
E_2^{i,-j}=H^i(\PP^n, \widetilde{\bK_j})\otimes A(-j))\Rightarrow R^{i-j}p_{2*}(\tau(\bK)).
\]
But 
the terms on this page all vanish except for
$$
E^{0,0}_2= H^{0}(\cO_{\PP^{n}})\otimes A = A,
$$
in homological degree 0, and 
$$
E^{n,-n-1}=H^{n}(\cO_{\PP^{n}}(-n-1)) \otimes  A(-n-1) = A(-n-1)
$$
in cohomological degree $n+(-n-1) = -1$.
Thus the complex $\bK \cdot \cE$ has the form
$$
\xymatrix{
A&A(-n-1)\ar[l]_-{ut^{n+1}}
}
$$
for some $u\in \kk$. But the complex $\FF$ has homology of finite length, so the homology of $\tau(\FF) \otimes p_{2}^{*}\cE$ is annihilated by a power of $t$, and thus $\FF\cdot \cE$ will also have homology annihilated by a power of $t$ (see Proposition~\ref{prop:exact} for more details).  It follows that $u\neq 0$, and $\FF\cdot \cE$ is quasi-isomorphic to the graded $A$-module $A/(t^{n})$, regarded as a complex concentrated in homological degree 0.
\end{example}

 Here is a more precise version of Theorem~\ref{thm:Phi}\eqref{thm:Phi:1}.
\begin{theorem}\label{thm:betti numbers of pairing}
The Betti numbers of $\FF\cdot \cE = \Phi(\FF,\cE)$ are given by the formula:
\[
\beta_{i,j}(\FF\cdot \cE)=\sum_{p-q=i}  \beta_{p,j}(\FF)\gamma_{q,-j}(\cE).
\]
In particular, the Betti table of $\FF\cdot \cE$ only depends on $\beta(\FF)$ and $\gamma(\cE)$.
\end{theorem}

\begin{proof}
Since $\Phi$ commutes with homological shifting, we may assume that $\FF$ is a minimal free complex supported entirely in nonnegative homological degrees, say $\FF=[\FF_0\gets \dots \gets \FF_p]$.  
We  compute $\FF\cdot \cE$ via a  spectral sequence.  First, we consider the double complex $C_{\bullet, \bullet}$ where $C_{i,\bullet}$ is the \v{C}ech resolution of $\left( \tau(\FF_i)\otimes_{\PP^n_A} (\cE\boxtimes \cO_{\AA^1})\right)$ on $\PP^n_A$ with respect to the standard \v{C}ech cover of $\mathbb P^n$.  If we represent all of the maps in $C_{\bullet, \bullet}$ with matrices, then all of the vertical maps (which are induced by the \v{C}ech resolutions) will involve bidegree $(0,0)$ elements, and all of the horizontal maps (which are induced by the maps in $\tau(\FF)$) will involve bihomogeneous elements that are strictly positive in both bidegrees.

Since $\Tot(C_{\bullet, \bullet})$ is a complex of flat $A$-modules that is quasi-isomorphic to $\FF\cdot \cE$, we can obtain the Betti numbers by computing $\Tor(\Tot(C_{\bullet, \bullet}), A/(t))$.  When we tensor by $A/(t)$, the vertical maps of $C_{\bullet, \bullet}$ are unchanged, but the horizontal maps all go to $0$.  
Hence, one spectral sequence degenerates, so the $i$th homology of $\Tot(C_{\bullet,\bullet})/(t)$ equals the sum 
\[
\HH^i(\Tot(C_{\bullet,\bullet})/(t))\cong \bigoplus_{j} \HH^j_{\text{vert}}(C_{i+j,\bullet})
\]

Now we compute the vertical homology.  
Recall that $\FF_i=\oplus_{j\in \ZZ} S(-j)^{\beta_{i,j}(\FF)}$.  For brevity, we use $\beta_{i,j}:=\beta_{i,j}(\FF)$, and we may thus write 
\[
\left( \tau(\FF_i)\otimes_{\PP^n_A} (\cE\boxtimes \cO_{\AA^1})\right)\cong \bigoplus_{j\in \ZZ} \cE(-j)^{\beta_{i,j}}\boxtimes A(-j).
\]
After taking the vertical homology of $C_{\bullet, \bullet}$, we then obtain:
\[
\xymatrix{
\bigoplus_{j} H^n(\PP^n, \cE(-j)^{\beta_{0,j}})\otimes A(-j)&\bigoplus_{j} H^n(\PP^n, \cE(-j)^{\beta_{1,j}})\otimes A(-j)\ar[l]&\dots\ar[l]\\
\vdots & \vdots&\\
\bigoplus_{j} H^1(\PP^n, \cE(-j)^{\beta_{0,j}})\otimes A(-j)&\bigoplus_{j} H^1(\PP^n, \cE(-j)^{\beta_{1,j}})\otimes A(-j)\ar[l]&\dots\ar[l]\\
\bigoplus_{j} H^0(\PP^n,  \cE(-j)^{\beta_{0,j}})\otimes A(-j)&\bigoplus_{j} H^0(\PP^n, \cE(-j)^{\beta_{1,j}})\otimes A(-j)\ar[l]&\dots\ar[l]
}
\]
We conclude that
\[
\Tor^i_A(\FF\cdot \cE, A/(t))\cong \HH^i(\Tot(C_{\bullet,\bullet})/(t))\cong \bigoplus_{p-q=i} \bigoplus_{j} H^q(\PP^n, \cO(-j)^{\beta_{p,j}}\otimes \cE)
\]
which proves the formula.
\end{proof}

\begin{example}\label{ex:1441}
Let $S=\kk[x,y,z]$ and let
\begin{equation}\label{eqn:intro ex}
\beta(\FF)=\begin{bmatrix} 1&-&-&-\\ -&-&-&-\\-&4&4&-\\-&4&4&-\\-&-&-&-\\-&-&-&1 \end{bmatrix}.
\end{equation}

No multiple of the Betti table $\beta(\FF)$ can equal the Betti table of a minimal, free complex with finite length homology. Indeed, such a complex would be a resolution,  so this case is covered by the theory of \cite{eis-schrey1}. But we can also see this directly from the Betti number formula of Theorem~\ref{thm:betti numbers of pairing}:

If $\cE$ is a rank $8$ supernatural bundle of type $(0,-8)$ then 
$\FF\cdot \cE$ would be a minimal complex of the form
\[
\FF\cdot \cE=\left[ \begin{matrix}A(-3)^{240}\\ \oplus \\A(-4)^{256}\end{matrix} \longleftarrow \begin{matrix}A(-4)^{256}\\\oplus \\ A(-5)^{240}\end{matrix}\right].
\]
Since the complex is minimal, the kernel of the map will contain a subsheaf of $A(-4)^{256}$ of rank at least $256-240=16$, and hence $\FF\cdot \cE$ cannot have finite length homology.  By Theorem~\ref{thm:Phi}\eqref{thm:Phi:2} it follows that no multiple of $\beta(\FF)$ can be the Betti table of a complex with finite length homology. 

By contrast, we shall see in Example~\ref{ex:1441bis} that 
 $10\beta(\FF)$ is the Betti table of a complex with 1-dimensional homology.
\end{example}

 Theorem~\ref{thm:betti numbers of pairing} implies that the values of $\Phi$ contain enough
 information to compute the Betti table of a complex $\FF$ or the cohomology
table of a sheaf $\cE$:
\begin{cor} Let $\FF\in \DD^b(S)$ and let $\cE\in \DD^b(\PP^n)$.
\begin{enumerate}
\item 
 $\beta_{i,j}(\FF) = \beta_{i,j}(\FF\cdot \cO_{\P^{n}}(j))$ for all $i,j$, where $\cO_{\PP^n}(j)$ is regarded as a complex concentrated in homological degree 0.
\item $h^{i}(\cE(j)) = \beta_{-i,-j}(S(j)\cdot \cE)$, for all $i,j$, where $S(j)$ is regarded as a complex concentrated in homological degree 0.
\end{enumerate}
\end{cor}

The following condition for $\FF\cdot \cE$ to have finite length homology will play a central role in our theory: 

\begin{prop}\label{prop:exact}
Let $\FF\in \DD^b(S)$ and $\cE\in \DD^b(\PP^n)$.  If $\widetilde \FF\otimes \cE$ is exact, then the homology of the complex $\Phi(\FF,\cE)$ has finite length.
\end{prop}

\begin{proof} It suffices to show that the homology of $\Phi(\FF,\cE)$ is annihilated by
a power of $t$. After inverting $t$ the map $\sigma$ becomes the usual inclusion $S\subset S[t,t^{-1}]$
composed with the invertible change of variables $x_{i}\mapsto x_{i}t$. Thus the complex 
$$
\Gbull:=\bigl(\widetilde\FF\otimes_{\sigma}\cO_{\PP^{n}\times \Spec A[t^{-1}]}\bigr)
\otimes_{\PP^{n}\times \Spec A[t^{-1}]}
\cE \cong \widetilde \FF \otimes \cE \otimes \cO_{\Spec A[t^{-1}]}
$$
has no homology. It follows by a spectral sequence computation that 
the complex $R\pi_{2*}\Gbull$ on $\Spec A[t^{-1}]$ has no homology. By flat base change,
this is equal to the restriction of $\Phi(\FF,\cE)$ on the open set $\Spec A[t^{-1}]\subseteq \Spec A$, and we see that the homology
of $\Phi(\FF,\cE)$ is annihilated by a power of $t$, as required.
\end{proof}

\section*{\underline{{Part II: Free Complexes with Homology}}}
\section{Decomposing Betti Tables}\label{sec:refined}
The main result of this section and the next two is Theorem~\ref{thm:extremal rays refined},
which extends the main decomposition results in \BS theory from \cites{eis-schrey1,boij-sod2}
to describe the cones of Betti tables of free complexes with  homology of at least a given codimension. More generally, we can treat certain cases where the homology modules have distinct codimensions. For this we make use of the following definitions.

A \defi{codimension sequence} is an indexed, nondecreasing,
doubly infinite sequence
$$
\cc=(\dots, c_{-1}, c_{0}, c_{1}, \dots )
$$
where  
$$
c_{i}\in \{\nothing\} \cup \{0,1,\dots,n+1\}\cup \{\infty\},
$$
for each $i$ and where we take the convention $\nothing<0$. As with degree sequences,
we will sometimes indicate the position of $c_{0}$ with $^{\circ}$.

We say that a
minimal graded free complex $\FF$ over $S$ (or the element of $\DD^{b}(S)$ that it represents) is \defi{compatible with $\cc$} if 
$
\codim \HH_i(\FF) \geq c_i, \text{ for all } i,
$
and if $\FF_i=0$ whenever $c_i=\nothing$.

Here are the cones of Betti tables and cohomology tables we will consider.
We write $\BBQ^{\cc}(S)$ for the subcone of $\VV$ spanned by $\beta(\FF)$ as $\FF$ ranges over all elements of $\DD^b(S)$ that are compatible with $\cc$.  For an integer $k\in \ZZ$, we write $\BBQ^k(S)$ to denote the cone corresponding to the codimension sequence $(\dots, k,k,\dots)$.
For any $k=0, \dots, n$, we define $\CQ^k(\PP^n)$ to be the subcone of $\WW$ spanned by all $\cE\in \DD^b(\PP^n)$ satisfying $\codim(\cE)\geq k$.

Suppose that $\cc$ is a codimension sequence. Let $\dd$ be a degree sequence 
of codimension $\ell\leq n+1$ (Definition~\ref{def:deg seq}), having the form
$$
\cdots -\infty, d_{k}, \cdots, d_{k+\ell}, \infty, \cdots
$$
with $d_{k},\dots, d_{k+\ell}\in \ZZ$.
We say that $\dd$ is \defi{compatible} with $\cc$ if $c_{k}\leq \ell \leq c_{k+1}$. 
For example, if the terms of $\cc$ are all equal to $k$, then the degree sequences
compatible with $\cc$ are precisely those of codimension $k$.

Also, it follows from the ``Lemme d'acyclicit\'e'' of Peskine and Szpiro~\cite{MR0374130}
that if a complex
$$
\FF= [\FF_{k}\leftarrow \cdots \leftarrow \FF_{k+\ell}\leftarrow 0],
$$
has $i$-th homology $\HH_{i}\FF$ of codimension $\geq\ell$ for $i>k$, then $\FF$ is actually a (shifted) resolution, and the module $\HH_{k}\FF$ that it resolves must have codimension $\leq \ell$. 
Thus, if $\FF$ is a pure complex of type $\dd$, and if both $\FF$ and $\dd$ are compatible with a codimension sequence $\cc$, then $\FF$ is a homologically
shifted resolution of a Cohen--Macaulay module of codimension $\ell$.

Our main result is a description of the structure of a simplicial fan for $\BBQ^{\cc}(S)$ in terms of  Betti tables of pure resolutions.  

\begin{thm}[Decomposition Theorem]\label{thm:extremal rays refined}
Fix a codimension sequence $c$.
If $\FF\in \DD^b(S)$  is compatible with $\cc$ then  $\beta(\FF)$ can be expressed
uniquely as a positive rational linear combination of the Betti tables of shifted pure resolutions of Cohen--Macaulay modules whose degree sequences form a chain and are compatible with $\cc$.

Thus the cone $\BBQ^{\cc}(S)$ is locally a simplicial fan, and there is a natural bijection:
\[
\left\{
\begin{matrix}
\text{Extremal rays of }\\
\text{the cone } \BBQ^{\cc}(S)
\end{matrix}
\right\}
\longleftrightarrow
\left\{
\begin{matrix}
\text{Shifted degree sequences }\\
\text{ compatible with $\cc$}
\end{matrix}
\right\},
\]
where the  degree sequence $\dd$ corresponds to the ray spanned by 
the Betti table of any shifted pure resolution of type $\dd$ of a Cohen-Macaulay module. 
\end{thm}

By the statement that $\BBQ^{\cc}(S)$ is locally a simplicial fan we mean that if we fix any finite dimensional subspace $\VV_{\text{fin}}\subseteq \VV$ defined by the vanishing of coordinate vectors, then the restricted cone $\VV_{\text{fin}}\cap \BBQ^{\cc}(S)$ is a simplicial fan (and, in particular,
is polyhedral). 

The proof of Theorem~\ref{thm:extremal rays refined} involves two essential steps:  In \S\ref{sec:A}, we provide a detailed description of cones of Betti tables on $A$.  Since $\DD^b(A)$ is the target of the pairing $\Phi$, these results will provide a base case for the theorem.  Then, in \S\ref{sec:refined proof}, 
we use the nonnegative functionals obtained by combining  \S\ref{sec:A} with Theorem~\ref{thm:Phi} to complete the proof of Theorem~\ref{thm:extremal rays refined}. 

\begin{cor}\label{cor:decompose refined}
If $\FF\in \DD^b(S)$ is compatible with a codimension sequence $\cc$,
then there exist Cohen-Macaulay graded $S$-modules $M^k$ of codimension $\geq \cc_k$ and nonnegative rational numbers $a_k$ such that
\[
\beta(\FF)=\sum_{k\in \ZZ} a_k\beta(M^k[k]).
\]
\end{cor}

\begin{example}[Decomposition Algorithm]\label{example:greedy decomposition}
The simplicial fan structure of $\BBQ^c(S)$ yields a decomposition algorithm parallel to that in~\cite[\S1]{eis-schrey1}, and we illustrate this via an example.  Let $S=\kk[x,y,z]$ and let $I=(x^2,xy,y^2,xz)$ and $J=(xy)$.  Let $\FF'$ be the minimal free resolution of $S/I$ and let $\FF''$ be the minimal free resolution of $S/J$.  We then consider the complex $\FF=\FF'\otimes \FF''$, so
\[
\beta(\FF)=
\begin{bmatrix}
1^\zp&-&-&-&-\\
-&{5}&{4}&{1}&-\\
-&-&4&4&{1}
\end{bmatrix}.
\]
We have that $\HH_i\FF=\Tor_i(S/I,S/J)$, and hence $\FF$ is compatible with the codimension sequence $\cc=(\emptyset, 2^\zp,2,\infty,\dots)$.    

To decompose a Betti table $\beta(\FF)$, we always choose the minimal degree sequence that is compatible with $\cc$ and that could possibly contribute to the Betti table.  Based on the partial order of degree sequences, this implies that the decomposition algorithm proceeds from the upper right corner to the lower left corner, and that we always zero out the rightmost column before shifting in homological degree.  In our example, since $\cc_2=\infty$ and $\cc_1<\infty$, we consider the top strand of $\FF$ starting in column $1$.  This yields the degree sequence $(\dots,\emptyset^\zp,2,3,4,6,\infty,\dots)$, which is compatible with $\cc$.  We then apply a greedy algorithm, subtracting as much of the corresponding pure diagram from $\beta(\FF)$ as possible, without making any entry negative.
\[
\begin{bmatrix}
1^\zp&-&-&-&-\\
-&\mathbf{5}&\mathbf{4}&\mathbf{1}&-\\
-&-&4&4&\mathbf{1}
\end{bmatrix}
-
\begin{bmatrix}
-^\zp&-&-&-&-\\
-&\frac{1}{2}&\frac{4}{3}&1&-\\
-&-&-&-&\frac{1}{6}
\end{bmatrix}
=
\begin{bmatrix}
1^\zp&-&-&-&-\\
-&\frac{9}{2}&\frac{8}{3}&-&-\\
-&-&4&4&\frac{5}{6}
\end{bmatrix}.
\]
The second step of the decomposition is similar:
\[
\begin{bmatrix}
1^\zp&-&-&-&-\\
-&\mathbf{\frac{9}{2}}&\mathbf{\frac{8}{3}}&-&-\\
-&-&4&\mathbf{4}&\mathbf{\frac{5}{6}}
\end{bmatrix}
-
\begin{bmatrix}
-^\zp&-&-&-&-\\
-&\frac{5}{6}&\frac{5}{3}&-&-\\
-&-&-&\frac{5}{3}&\frac{5}{6}
\end{bmatrix}
=
\begin{bmatrix}
1^\zp&-&-&-&-\\
-&\frac{11}{3}&1&-&-\\
-&-&4&\frac{7}{3}&-
\end{bmatrix}
\]
After the second step, we have zeroed out the last column.  Since $c_1=2$, we next consider degree sequences of codimension $2$ that start in column.  The final decomposition is:
\begin{align*}
\beta(\FF)=&
\begin{bmatrix}
-^\zp&-&-&-&-\\
-&\frac{1}{2}&\frac{4}{3}&1&-\\
-&-&-&-&\frac{1}{6}
\end{bmatrix}
+
\begin{bmatrix}
-^\zp&-&-&-&-\\
-&\frac{5}{6}&\frac{5}{3}&-&-\\
-&-&-&\frac{5}{3}&\frac{5}{6}
\end{bmatrix}
+
\begin{bmatrix}
-^\zp&-&-&-&-\\
-&\frac{2}{3}&1&-&-\\
-&-&-&\frac{1}{3}&-
\end{bmatrix}\\
&
+
\begin{bmatrix}
-^\zp&-&-&-&-\\
-&1&-&-&-\\
-&-&3&2&-
\end{bmatrix}
+
\begin{bmatrix}
1^\zp&-&-&-&-\\
-&2&-&-&-\\
-&-&1&-&-
\end{bmatrix},
\end{align*}
with the corresponding chain of degree sequences
\begin{align*}
(\dots,\emptyset^\zp,2,3,4,6,\infty,\dots)<(\dots,\emptyset^\zp,2,3,5,6,\infty,\dots)<(\dots,\emptyset^\zp,2,3,5,\infty\dots)
\\
<(\dots,\emptyset^\zp,2,4,5,\infty,\dots)<(\dots,0^\zp,2,4,\infty,\dots).
\end{align*}
At the final step of this decomposition, we shift over by a column, using a degree sequence $d$ where $d_0\in\ZZ$.  This is because we have run out of degree sequences with $d_0=-\infty$ that are compatible with $\cc$.  In particular, the degree sequence $(\dots, -\infty^\zp,2,4,\infty,\dots)$ only has codimension $1$, and is thus not compatible with $\cc$.
\end{example}

\begin{example}\label{example:many decompositions}
For any complex $\FF\in \DD^b(S)$, there exist different values of $\cc$ that are compatible with $\FF$, and the decomposition of $\beta(\FF)$ induced by Theorem~\ref{thm:extremal rays refined} may depend on the choice of $\cc$.   For instance, let $S=\kk[x,y,z]$, $I=(x^2,xy,y^2,xz)$, and let $\FF$ be the minimal free resolution of $S/I$.  Since $\FF$ is a resolution, we may choose $\cc=(\dots, \nothing, 2^\zp, \infty, \dots)$ and then we decompose
\[
\beta(\FF)=\begin{bmatrix}
1^\zp&-&-&-\\
-&4&4&1
\end{bmatrix}
=
\frac{1}{3}
\begin{bmatrix}
1^\zp&-&-&-\\
-&6&8&3
\end{bmatrix}
+\frac{2}{3}
\begin{bmatrix}
1^\zp&-&-\\
-&3&2
\end{bmatrix}.
\]
$\FF$ is also compatible with $\cc'=(\dots, 2, 2^\zp, 2, \dots)$.  In this case, we obtain the decomposition:
\[
\beta(\FF)=
\begin{bmatrix}
1^\zp&-&-&-\\
-&4&4&1
\end{bmatrix}
=\begin{bmatrix}1^\zp&-&-\\-&3&2\end{bmatrix}
+
\begin{bmatrix}
-^\zp&-&-&-\\
-&1&2&1
\end{bmatrix}
\] 
This second decomposition is stable under taking a hyperplane section.  Namely, set $S':=S/(\ell)$, where $\ell$ is a generic linear form, and let $\FF'$ be the restriction of $\FF$ to $S'$.  Since $\text{depth}(S/I)=0$, the complex $\FF'$ is not a resolution, but it is still compatible with $\cc'$, and hence the second decomposition still holds for $\FF'$.
\end{example}

\begin{example}[Resolutions]\label{ex:resolutions}
 If  $\cc=(\dots, \nothing, n+1^\zp, \infty, \dots)$, then a complex $\FF$ is compatible with $\cc$ if and only if $\FF$ is the minimal free resolution of an $S$-module of finite length, and we recover \cite[Theorem~0.2]{eis-schrey1}. More generally, the resolution of any module is compatible with
$\cc=(\dots, \nothing, 0^\zp, \infty, \dots)$, and we recover the main results of \cite{boij-sod2}. 
\end{example}

\begin{example}[$\cO_{\PP^n}$-resolutions]\label{ex:sheaf resolutions}
If $\cc=(\dots,\nothing,0^{\zp},n+1,n+1,\dots)$ then a complex $\FF$ is compatible with $\cc$ if and only if $\widetilde{\FF}$ resolves a coherent $\cO_{\PP^n}$-module.
\end{example}

\begin{example}[Approximate resolutions]
Let $\cc=(\dots,\nothing,0^\zp,n-1,n-1,\dots,)$.   Let $\FF\in \DD^b(S)$ be compatible with $\cc$ and let $M$ be the $0$th homology module of $\FF$.  We note that $\FF$ provides an approximate resolution of $M$, in the sense that the homology of $[\widetilde{\FF}\to \widetilde{M}]$ has dimension at most $1$.  Such approximate resolutions play a key role in \cite[Lemma~1.6]{gruson-lazarsfeld-peskine}, where they are used to bound the Castelnuovo--Mumford regularity of $\widetilde{M}$.
\end{example}

\begin{example}\label{ex:1441bis}
Returning to the complex $\FF$ of Example~\ref{ex:1441}
we see from the following decomposition that, although $\beta(\FF)\notin \BBQ^3(S),$ it does lie in $\BBQ^2(S)$:
\[
\begin{bmatrix} 1^\zp&-&-&-\\ -&-&-&-\\-&4&4&-\\-&4&4&-\\-&-&-&-\\-&-&-&1 \end{bmatrix}
=
\begin{bmatrix} -^\zp&-&-&-\\ -&-&-&-\\-&\frac{16}{5}&4&-\\-&-&-&-\\-&-&-&-\\-&-&-&\frac{4}{5} \end{bmatrix}
+
\begin{bmatrix} -^\zp&-&-&-\\ -&-&-&-\\-&\frac{3}{10}&-&-\\-&-&\frac{1}{2}&-\\-&-&-&-\\-&-&-&\frac{1}{5} \end{bmatrix}
+
\begin{bmatrix} \frac{1}{5}^\zp&-&-&-\\ -&-&-&-\\-&\frac{1}{2}&-&-\\-&-&\frac{3}{10}&-\\-&-&-&-\\-&-&-&- \end{bmatrix}
+
\begin{bmatrix} \frac{4}{5}^\zp&-&-&-\\ -&-&-&-\\-&-&-&-\\-&4&\frac{16}{5}&-\\-&-&-&-\\-&-&-&- \end{bmatrix}.
\]
Up to scalar multiple, $\beta(\FF)$ thus equals the Betti table of a complex whose homology modules have codimension $2$.  Equivalently, $\beta(\FF)$ 
equals, up to scalar multiple, the Betti table of a complex over $\kk[x,y]$ with finite length homology.
\end{example}

\begin{example}
If we think of a pure resolution of a module of finite length as a complex with homology of dimension at most 1, then we recover the fact, first observed by Mats Boij, that a pure table of type $d=(d_0, \dots, d_{n+1})$ can be written as the sum of a pure table of type $(d_0, \dots, d_n)$ and a  pure table of type $(d_1, \dots, d_{n+1})$.  For instance
\[
\begin{bmatrix}
1^{\zp}&-&-&-\\
-&5&5&-\\
-&-&-&1
\end{bmatrix}
=
\begin{bmatrix}
1^{\zp}&-&-&-\\
-&3&2&-\\
-&-&-&-
\end{bmatrix}
+
\begin{bmatrix}
-^{\zp}&-&-&-\\
-&2&3&-\\
-&-&-&1
\end{bmatrix}.
\]
We may interpret this decomposition as follows.  Let $M$ be an $S=\kk[x,y,z]$-module whose Betti table is the diagram on the left. If $\ell$ is a generic linear form (assuming $\kk$ is infinite), then as $S/\ell$-modules, the Betti tables of $\Tor^0_S(M,S/\ell)$ and $\Tor^1(M,S/\ell)$ correspond to the other Betti tables above.
\end{example}

\section{Complexes on $\kk[t]$}\label{sec:A}
We first prove Theorem~\ref{thm:extremal rays refined} for cones of the form $\BBQ^{\cc}(A)$ (which is the $n=0$ case of the theorem).
Since the output of $\Phi$ is a complex on $A$, this provides a base case for the theorem.  In addition, we provide a halfspace description of each such cone.

Recall that $\UU= \oplus_{i\in \ZZ} \oplus_{j\in \ZZ}\QQ$  denotes the vector space containing the cones $\BBQ^{\cc}(A)$. 

\begin{defn}\label{defn:chi}
Given $\Gbull\in \DD^b(A)$, we define $\chi_{i,j}(\Gbull)=\chi_{i,j}(\beta(\Gbull))$ to be the dot product of the Betti table $\beta(\Gbull)$ with
\begin{equation}\label{eqn:chi}
\chi_{i,j}:=
\left|
\begin{array}{cccccccc}
 & \vdots& &\multicolumn{1}{|c}{\vdots}&& \vdots&&\\
\dots&0&0&\multicolumn{1}{|c}{1}&-1&1&-1&\dots\\
\dots&0&0&\multicolumn{1}{|c}{1}&-1&1&-1&\dots\\
\dots&0&0&\multicolumn{1}{|c}{\mathbf{1}}&-1&1&-1&\dots\\ \cline{4-5}
\dots&0&0&0&0&\multicolumn{1}{|c}{1}&-1&\dots\\
\dots&0&0&0&0&\multicolumn{1}{|c}{1}&-1&\dots\\
& \vdots&&\vdots&& \multicolumn{1}{|c}{\vdots}&&\\
\end{array}
\right|
\end{equation}
where the boldface $1$ corresponds to $\beta_{i,j}$; that is, the boldface $1$ is in column $i$ and row $i-j$. 
\end{defn}
The line that snakes through
the table indicates how $\chi_{i,j}$ separates a Betti table into two regions.  In the upper region, this is simply computing an Euler characteristic, and in the lower region, it is dot product with the zero matrix. 

We define $\DD^b(A)_{\text{tor}}$ as the subcategory of $\DD^b(A)$ consisting of complexes whose homology is torsion.
The usefulness of the functionals $\chi_{i,j}$ comes from their positivity properties.

\begin{prop}\label{lem:chi nonneg}
If $\Gbull \in \DD^b(A)_{\text{tor}}$, then $\chi_{i,j}(\Gbull)\geq 0$.
\end{prop}

\begin{proof} Because $A$ has global dimension 1, every minimal free complex
over $A$ is isomorphic to the direct sum of the resolutions of its homology modules, 
and any torsion $A$-module is a direct sum of modules of the form
 $A(-p)/t^q$. Thus it suffices to compute $\chi_{i,j}\Gbull$, where
$
\Gbull
$
is, up to a homological shift, the two term complex
$A(-p)\leftarrow A(-p-q)$.  It is then straightforward to verify that $\chi_{i,j}\Gbull$ is 0 or 1,
depending on the values of $i,j,p,q$.
\end{proof}

By composing the $\chi_{i,j}$ with $\Phi$, we obtain a positivity result generalizing those in \cite[\S4]{eis-schrey1}.

\begin{cor}\label{thm:categorified}\label{thm:categorified:1}  Let $\FF\in \DD^b(S)$ and $\cE\in \DD^b(\PP^n)$. If  $\codim(\FF)+\codim(\cE)\geq n+1$, then
\[
\chi_{i,j}(\FF\cdot \cE)\geq 0
\]
\end{cor}

\begin{proof}
The number $\chi_{i,j}(\FF\cdot \cE)$ only depends on the Betti table of $\FF\cdot \cE$, which by Theorem~\ref{thm:Phi}\eqref{thm:Phi:1} only depends on $\beta(\FF)$ and $\gamma(\cE)$.  We may thus assume that $\kk$ is an infinite field, and we may replace $\cE$ by a general $\GL_{n+1}$ translate, since this does not affect the cohomology table of $\cE$.  By~\cite[Theorem, p.\ 335]{miller-speyer}, we  conclude $\widetilde{\FF}\otimes \cE$ is exact.  Hence, by Theorem~\ref{thm:Phi}\eqref{thm:Phi:2}, it follows that $\FF\cdot \cE$ lies in $\DD^b(A)_{\text{tor}}$.  Thus Proposition~\ref{lem:chi nonneg} yields the desired nonnegativity.
\end{proof}

\begin{example}\label{ex:sup2}
On $\PP^2$, let $\cE$ be a supernatural bundle of type $(1,-3)$ and rank $2$.  
We consider the functional $\beta(\FF)\mapsto \chi_{0,0}(\FF\cdot \cE)$.  This functional is given by the dot product of $\beta(\FF)$ with
\[
\chi_{0,0}(-\cdot \cE)=
\left|
\begin{array}{ccccccc}
&\multicolumn{1}{|c}{\vdots}&\vdots&\vdots&\vdots&&\\
\dots&\multicolumn{1}{|c}{12}&-5&0&3&\dots\\
\dots&\multicolumn{1}{|c}{5}&0&-3&4&\dots\\ \cline{2-2}
\dots&{0}&\multicolumn{1}{|c}{3}&-4&3&\dots\\ \cline{3-5}
\dots&0^\zp&0&0&{0}&\multicolumn{1}{|c}{\dots}\\
\dots&0&0&0&0&\multicolumn{1}{|c}{\dots}\\
\dots&0&0&0&0&\multicolumn{1}{|c}{\dots}\\
\dots&0&0&0&0&\multicolumn{1}{|c}{\dots}\\
&\vdots&\vdots&\vdots&\vdots&\multicolumn{1}{|c}{}&\\
\end{array}
\right|.
\]
The line that snakes through the table corresponds to the dividing line in the definition of $\chi_{0,0}$. The shape of the line has changed, due to the homological shifts introduced by the cohomology table of $\cE$.  Note that this recovers the facet equation $\delta$ from~\cite[\S3]{eis-schrey-icm}.
\end{example}

We will 
also use the total Euler characteristic functional $\chi(\FF):=(\FF\mapsto \sum_{i,j}(-1)^i \beta_{i,j}(\FF))$.  

\begin{cor}\label{cor:dualconeA refined}
Let $\cc = (\dots,c_{i},\dots) = (\dots,\nothing,0\dots,0,1\dots,1,\infty, \dots)$ be a codimension sequence for $A$ and let $u\in \UU$.
A point $u\in \UU$ lies in the cone $\BBQ^{\cc}(A)$ if and only if the following equalities and inequalities hold:
\begin{enumerate}
	\item $\beta_{i,j}(u)=0$ for all $i,j$ such that $c_i=\nothing$.
	\item $\beta_{i,j}(u)\geq 0$ for all $i,j$.
	\item  $\chi_{i,j}(u)\geq 0$ for all $i,j$ such that $c_{i+1}\geq 1$.
	\item If there does not exist $i$ such that $c_i=0$, then $\chi(u)=0$.
\end{enumerate}
\end{cor}
\begin{proof}[Proof of Theorem~\ref{thm:extremal rays refined} for $\BBQ^{\cc}(A)$ and of Corollary~\ref{cor:dualconeA refined}]
Since $A$ has global dimension $1$, any complex $\Gbull\in \DD^b(A)$ is quasi-isomorphic to its homology.  We may thus write $\Gbull$ as a direct sum of shifted indecomposable modules.  Recall that the indecomposable graded $A$-modules are twists of $A$ or $A/(x^p)$ for $p\geq 0$.  Hence, each individual module appearing in that decomposition corresponds to the homology of a shifted pure resolution, and conversely each degree sequence over $A$ corresponds to a unique shifted indecomposable module.  The extremal ray description of $\BBQ^{\cc}(A)$ then immediately follows.

We next check that each functional in (1) and (4) vanishes on $\BBQ^{\cc}(A)$.  For (1), if $c_i=\nothing$ this simply follows from the fact that $\Gbull$ is compatible with $\cc$.  For (4), if there does not exist $i$ such that $c_i=0$, then $\Gbull$ is compatible with $\cc$ only if $\Gbull$ has finite length homology, and hence $\chi(\Gbull)=0$.  

We now claim that the other functionals in Corollary~\ref{cor:dualconeA refined} are nonnegative.  The functionals from (2) are obviously nonnegative, so it suffices to check that $\chi_{i,j}(\Gbull)\geq 0$ when $\Gbull$ is compatible with $\cc$ and $c_i\geq 1$.  By the extremal ray description, it suffices to consider the case when $\Gbull$ is a pure resolution.  If $\Gbull$ resolves a finite length module, then Proposition~\ref{lem:chi nonneg} implies that $\chi_{i,j}(\Gbull)\geq 0$.  On the other hand, if $\Gbull=A(-p)[q]$ then, since $c_{i+1}=1$, the homology of $\Gbull$ lies entirely in homological degree $\leq i$; so $q\leq i$ and hence $\chi_{i,j}(\Gbull)\geq 0$.

To obtain the results about the simplicial structure of $\BBQ^{\cc}(A)$, we must show that if a point $u\in \UU$ satisfies the inequalities in Corollary~\ref{cor:dualconeA refined}, then we may write $u$ uniquely as a sum of pure tables whose degree sequences form a chain.  It suffices to consider points $u\in \UU$ whose entries are all integral.  Since the entries of $u$ must be nonnegative, we may induct on the sum of all of the entries of $u$.  When all entries of $u$ are zero, then $u$ is the empty sum of pure diagrams, and this provides our base case.

Otherwise, $u$ has some nonzero entry.  Our goal is to produce some new diagram $u'$ from $u$ on which we can apply the induction hypothesis.  Choose $(s,t)$ so that $u_{s,t}$ is the top nonzero entry in the rightmost nonzero column of $u$.
If $c_s=0$, then $\beta(A(-t)[-s])$ is compatible with $\cc$, and we set $u'=u-\beta(A(-t)[-s])$.  On the other hand, if $c_s>0$, then we choose $r$ so that $u_{s-1,r}$ is the top nonzero entry in column $s-1$.  We claim that $r<t$, so that $u$ has the form:
\[
u=
\left[
\begin{array}{ccccccc}
 & \vdots& \multicolumn{1}{|c}{\vdots}&\vdots& \vdots&&\\
\dots&u_{s-2,r-2}&\multicolumn{1}{|c}{0}&0&0&\dots\\
\dots&u_{s-2,r-1}&\multicolumn{1}{|c}{u_{s-1,r}}&0&0&\dots\\
 & \vdots& \multicolumn{1}{|c}{\vdots}&\vdots& \vdots&&\\
\dots&u_{s-2,t-3}&\multicolumn{1}{|c}{u_{s-1,t-2}}&0&0&\dots\\
\dots&u_{s-2,t-2}&\multicolumn{1}{|c}{u_{s-1,t-1}}&u_{s,t}&0&\dots\\ \cline{3-4}
\dots&u_{s-2,t-1}&u_{s-1,t}&u_{a,t+1}&\multicolumn{1}{|c}{0}&\dots\\
& \vdots&\vdots&\vdots& \multicolumn{1}{|c}{\vdots}&&\\
\end{array}\right].
\]
Since $c_s>0$ we have $\chi_{s-1,t-1}(u)\geq 0$, and it follows that $r<t$ as claimed.  When $c_s>0$, we thus set
\[
u':=u-\beta(A(-r)[-s+1]/t^{t-r}).
\]

Since $u$ satisfies the inequalities in Corollary~\ref{cor:dualconeA refined}, one may verify directly that $u'$ (in either of the cases above) also satisfies these inequalities.  Hence, the induction hypothesis guarantees that we can write $u'$ uniquely as a sum of pure tables whose degree sequences form a chain.  Further, the degree sequence corresponding to $u-u'$ is less than or equal to any degree sequence that could possibly arise in the decomposition of $u'$, and hence we can use the decomposition of $u'$ to conclude that $u$ decomposes uniquely as a sum of pure tables whose degree sequences form a chain.
\end{proof}

\section{Proof of the Decomposition Theorem}\label{sec:refined proof}\label{sec:general case}

We begin with a lemma which is like a refined version of Theorem~\ref{thm:Phi}\eqref{thm:Phi:2}.  

\begin{lemma}\label{lem:refined positivity}
Fix the codimension sequence $\cc'$ where:
\[
\cc'_i:=\begin{cases}
1&\text{ if }i\geq 1\\
0&\text{ if } i<0. 
\end{cases}
\]
Let $\cc$ be any codimension sequence and write $k:=\cc_1$.  The  pairing $\Phi$ induces a map of cones:
\[
\BBQ^{\cc}(S)\times \CQ^{n+1-k}(\PP^n)\to \BBQ^{\cc'}(A)
\]
given by $(\beta(\FF), \gamma(\cE))\mapsto \beta(\FF\cdot \cE)$.
\end{lemma}
\begin{proof}
Let $\FF\in \DD^b(S)$ be compatible with $\cc$ and let $\cE\in \DD^{b}(\PP^n)$ have codimension $\geq n+1-k$.  By Theorem~\ref{thm:Phi}\eqref{thm:Phi:1}, the Betti table of $\FF\cdot \cE$ only depends on the Betti table of $\FF$ and on the cohomology table of $\cE$.  Since extending the ground field does not affect any of these tables, we may assume $\kk$ is infinite.  Fixing $r$, we may replace $\cE$ by a general $\GL_{n+1}(\kk)$-translate and apply~\cite[Theorem, p.\ 335]{miller-speyer} to obtain that $\HH_r \widetilde{\FF}$ and $\cE$ are homologically transverse, that is:
\begin{itemize}
	\item $\Tor_i(\HH_r\widetilde{\FF},\cE)=0 \text{ for } i>0$, and
	\item $\codim ((\HH_r\widetilde{\FF})\otimes \cE) \geq \min\{n+1, \codim \HH_r\widetilde{\FF}+\codim \cE\}.$
\end{itemize}
In fact, since $\FF$ is bounded, a general translate of $\cE$ will be homologically transverse to all of the $\HH_r\widetilde{\FF}$ simultaneously.  
Since $c_r\geq k$ for all $r\geq 1$, transversality implies that $(\HH_r\widetilde{\FF})\otimes \cE=0$ whenever $r\geq 1$.

We now consider the spectral sequence
\[
E^2_{r,q}=\Tor_q(\HH_r\widetilde{\FF},\cE)\Rightarrow \Tor_{r+q}(\widetilde{\FF},\cE).
\]
The conclusion of the previous paragraph shows that $E^2_{r,q}=0$ whenever $r+q\geq 0$, and we thus conclude that 
$\widetilde{\FF}\otimes \cE$ is exact in homological degrees $\geq 1$.

By flat base change with respect to $U:=\AA^1\setminus \{0\} \subseteq \AA^1$, the complex $\Phi(\FF,\cE)\otimes_{\cO_{\AA^1}} \cO_{U}$ is quasi-isomorphic to $Rp_{2*}(\widetilde{\FF}\otimes \cE)\otimes_{\kk} \cO_{U}$.  It follows that $\Phi(\FF,\cE)$ is compatible with $\cc'$ if and only if $R^sp_{2*}(\widetilde{\FF}\otimes \cE)=0$ for $s\geq 1$.  This in turn follows immediately from the above computation and from the hypercohomology spectral sequence
$
E_2^{r,q}=R^{r}p_{2*}(\HH_{-q}(\widetilde{\FF}\otimes \cE))\Rightarrow R^{q+r}p_{2*}(\widetilde{\FF}\otimes \cE).
$
\end{proof}

\begin{remark}
Fix $\FF$ and $\cE$ such that $\codim(\FF)+\codim(\cE)\geq n+1$.  If $\FF$ and $\cE$ are not transverse, then it may happen that $\FF\cdot \cE$ fails to have finite length homology. However, the lemma implies that $\FF\cdot \cE$ has the same Betti table as a complex of $A$-modules with finite length homology.  
\end{remark}

The following lemma will be used repeatedly in our proof of Theorem~\ref{thm:extremal rays refined}.
\begin{lemma}\label{lem:pure and supernatural}
Let $\dd$ be the codimension $n+1$ degree sequence corresponding to a pure complex $\FF_\dd$, and let $f=(f_1>\dots >f_s)$ be the root sequence corresponding to a supernatural sheaf $\cE_f$.  Then
$
\beta_{i,j}(\FF_\dd\cdot \cE_f)\ne 0
$
if and only if $d_\ell=j$ where $f_{\ell-i}>-d_{\ell}>f_{\ell-i+1}$.\end{lemma}
\begin{proof}
By twisting and shifting $\FF_\dd$, we may without loss of generality assume that $i=j=0$.  By Theorem~\ref{thm:betti numbers of pairing},
 $\beta_{0,0}(\FF_\dd\cdot \cE_f)$ can only be nonzero if $0$ is one of the entries of $\dd$.  For this proof, we also set $f_0=\infty$ and $f_{s+1}=-\infty$. We may now assume that $d_{\ell}=0$ for a unique $\ell$, and that $f_m>0>f_{m+1}$ for a unique $m$.
Since $\cE_f$ is supernatural, it follows that $\gamma_{q,0}(\cE_f)\ne 0 \iff q=m$.  Hence, we may apply Theorem~\ref{thm:betti numbers of pairing} to conclude
\[
\beta_{0,0}(\FF_\dd\cdot \cE_f)\ne 0 \iff m=\ell,
\]
implying the lemma.
\end{proof}

We now tackle the proof of Theorem~\ref{thm:extremal rays refined}, which can be roughly separated into two pieces:
\begin{itemize}
	\item  A careful analysis of the combinatorics of the relevant simplicial fans, culminating in a classification of the boundary facets.  This argument is combinatorial in nature, and it is concluded in Lemma~\ref{lem:combinatorial}.
	\item  A proof that each of the boundary facets in Lemma~\ref{lem:combinatorial} induces a non-negative functional on any relevant Betti table.  This argument is algebraic/algebro-geometric in nature and relies heavily on our introduction of the functor $\Phi$.
\end{itemize}

We introduce some notation which will be useful for the combinatorial part of the proof.   Throughout the remainder of this section, we take the convention that, for any degree sequence $\dd$, $\FF_{\dd}$ denotes some pure complex of type $\dd$ and with same codimension as $\dd$.  
Fix $(\delta,\epsilon)\in \ZZ^2$ with $\epsilon-\delta\geq n+1$ and let $P_{(\delta,\epsilon)}$ be the subposet consisting of all degree sequences between its minimal element 
\[\bordermatrix{&&&\delta&&&\epsilon-n-1&&\epsilon &&\cr
              d_{\min}= &\dots&-\infty,&-\infty,&\dots&-\infty,&-n-1,&\dots &0, &\infty,&\dots\cr},
              \]
              which has codimension $n+1$, and its maximal element
\[\bordermatrix{&&\delta&&\delta+c_\delta&&&\epsilon &&\cr
              d_{max}= &\dots&0,&\dots&c_{\delta}+1,&\infty,&\dots &\infty,&\infty,&\dots\cr},\]
              which has codimension $c_{\delta}$.
Let $\Sigma_{(\delta,\epsilon)}$ be the convex cone spanned by the pure diagrams $\beta(\FF_{\dd})$, as $\dd$ ranges over the poset $P_{(\delta,\epsilon)}$.  We may apply the proof of~\cite[Proposition~2.9]{boij-sod1} to conclude that $\Sigma_{(\delta,\epsilon)}$ has the structure of a simplicial fan, with simplices corresponding to chains in $P_{(\delta,\epsilon)}$.

We also set $\VV_{(\delta,\epsilon)}$ to be the subspace of $\VV$ defined by
\[
\VV_{(\delta,\epsilon)}:=\bigoplus_{i=\delta}^\epsilon \bigoplus_{j=-\epsilon +i}^{-\delta+i} \QQ.
\]
Visually, $\VV_{(\delta,\epsilon)}$ is vector space of Betti tables that fit inside the box:
\[
\begin{bmatrix}
\beta_{\delta,\delta-\epsilon}&\beta_{\delta+1,\delta-\epsilon+1}&\dots&\beta_{\epsilon,0}\\
\vdots & &\ddots & \vdots\\
\beta_{\delta,-1}&\beta_{\delta+1,0}&\dots&\beta_{\epsilon,\epsilon-\delta-1}\\
\beta_{\delta,0}&\beta_{\delta+1,1}&\dots&\beta_{\epsilon,\epsilon-\delta}\\
\end{bmatrix}
\]

Our proof of Theorem~\ref{thm:extremal rays refined} relies on a classification of the halfspaces defining $\Sigma_{(\delta,\epsilon)}$.  Given any maximal chain in $P_{(\delta,\epsilon)}$, the corresponding pure Betti tables span the codimension $c_{\delta}$ subspace of $\VV_{(\delta,\epsilon)}$ cut out by the vanishing of the first $c_{\delta}$ Herzog--K\"uhl equations (this follows by adapting the argument from the proof of \cite[Proposition~1]{boij-sod2}).  
Thus, inside of its span, $\Sigma_{(\delta,\epsilon)}$ is a full-dimensional, equidimensional simplicial fan.  
As discussed in~\cite[Appendix]{bbeg}, we may thus talk about boundary facets of $\Sigma_{(\delta,\epsilon)}$.
These correspond to submaximal chains of $P_{(\delta,\epsilon)}$ and, as in \cite[Proposition~2.12]{boij-sod1}, each halfspace (with the exception of case (i) below) is entirely determined by the omitted element $\dd$ and its two adjacent neighbors $\dd'$ and $\dd''$.  We thus refer to such a halfspace by the triplet $\dd'<\dd<\dd''$.  

\begin{lemma}\label{lem:combinatorial}
The different types of boundary facets of $\Sigma_{(\delta,\epsilon)}$ that arise are the following.  (In the examples given in this list, we always assume that $\delta<0$.)
\begin{enumerate}[(i)]

	\item A chain where we omit either the maximal or minimal element of $P_{(\delta,\epsilon)}$.

	\item A chain where $d'_i<d_i<d''_i$ for some $i$.  This can arise in several different ways depending on whether a homological/codimension shift occurs.
	One such example is any chain
	\[
(\dots,1^\zp,\dots) <(\dots, 2^\zp,\dots) <(\dots,3^\zp,\dots).
	\]

	\item A chain where $\dd', \dd,$ and $\dd''$ differ by $1$ in adjacent positions.  For this to yield a boundary facet, we must have $d''_i=d_i+1=d'_i+1$, $d''_{i+1}=d_{i+1}=d'_{i+1}+1$ and $d'_i+1=d'_{i+1}$.  For example, the chain
			\[
(\dots,, 0,1^\zp,\dots) <(\dots, 0,2^\zp,\dots) <(\dots, 1,2^\zp,\dots) .
			\]
	\item A chain where $\dd'<\dd<\dd''$ consists of either two homological shifts or a homological shift and a codimension shift (in either order).  
For example, the chain:
		\[
		(\dots, -\infty^\zp,0,2-\delta,3-\delta,\dots)<(\dots, -\epsilon^\zp,0,2-\delta,\infty,\dots)<(\dots, -\epsilon^\zp,0,\infty,\infty,\dots).
		\]

	\item A chain where $\dd'<\dd<\dd''$ consists of	 two codimension shifts.
One such example is the chain:
\[
(\dots, -\infty^\zp,0,2-\delta,3-\delta,\dots)<(\dots, -\infty^\zp,0,2-\delta,\infty,\dots)<(\dots, -\infty^\zp,0,\infty,\infty,\dots).
\]
\end{enumerate} 
\end{lemma}
\begin{proof}
This proof differs in a significant way from similar proofs in \cites{eis-schrey1,boij-sod2}.  This is because there are now three ways that adjacent elements $\dd<\dd'$ can arise in the poset $P_{(\delta,\epsilon)}$.   The first possibility is simply that $\dd'$ is obtained from $\dd$ by adding $1$ to a single entry, as in:
\[
(\dots,3^\circ,\dots)<(\dots,4^\circ,\dots).
\]

The second possibility is when the finite entries of $\dd$ and $\dd'$ lie in different homological positions, and this can only arise in a specified way.  Namely, let $b$ be the codimension of $\dd$.  For this case to occur, there exists a unique column $i$ such that: $d_j=d'_j$ for $j\notin \{i, i-b-1\}$; $d_i=i-\delta$ and $d'_i=\infty$; and $d_{i-b-1}=-\infty$ and $d'_{i-b-1}=i-b-1-\epsilon$.  For example, if $b=2$ and $i=0$, then we could have:
\[
(\dots, -\infty, -2, 0, -\delta^\circ, \dots)<(\dots, -3-\epsilon,-2, 0, \infty^\circ, \dots).
\]
We refer to this as a \defi{homological shift from column $i$}.

The third possibility is that $\dd$ and $\dd'$ correspond to degree sequences of different codimensions.  Again, this can only arise a specific manner.  Namely, we must have $d_{i}=i-\delta$ and $d'_i=\infty$ for some $i$.  
For example, if $i=0$, then we could have:
\[
(\dots, -\infty, -2, 0, -\delta^\circ, \dots)<(\dots, -\infty,-2, 0, \infty^\circ, \dots).
\]
We refer this as a \defi{codimension shift}.  

The proof of the lemma then follows by a case-by-case analysis of the various combinations of adding $1$ to a single entry, homological shifts, and codimension shift.  This analysis is involved though elementary, and we omit the details.
\end{proof}

\begin{proof}[Proof of Theorem~\ref{thm:extremal rays refined}]
Fix $(\delta,\epsilon)\in \ZZ^2$ with $\epsilon-\delta\geq n+1$.  We first let $\widetilde{P}_{(\delta,\epsilon)}$ be the poset of degree sequences that are compatible with $\cc$ and whose corresponding rays lie in $\VV_{(\delta,\epsilon)}$.  Since we can ignore any column $i$ with $c_i=\nothing$, we may, without loss of generality, assume that $c_\delta \geq 0$.  By possibly shrinking $n$ (which will not affect the rest of the proof), we may also assume that the minimal degree sequence in $\widetilde{P}_{(\delta,\epsilon)}$ has codimension $n+1$.  Lastly, we may assume that $c_{\epsilon-n-1}<\infty$, as else we could replace $\epsilon$ by $\epsilon-1$.
This implies that $\widetilde{P}_{(\delta,\epsilon)}$ equals the poset $P_{(\delta,\epsilon)}$ defined above in terms of $d_{\min}$ and $d_{\max}$.

We define $D_{(\delta,\epsilon)}$ to be the intersection of the halfspaces corresponding to all boundary facets of $\Sigma_{(\delta,\epsilon)}$.  We clearly have $D_{(\delta,\epsilon)} \subseteq \Sigma_{(\delta,\epsilon)}$.
The existence of pure resolutions~\cite[Theorem~0.1]{eis-schrey1} implies that, for any degree sequence from $P_{(\delta,\epsilon)}$, the corresponding ray lies in $\BBQ^{\cc}(S)$.  Hence $\Sigma_{(\delta,\epsilon)}$ is a subcone of $\BBQ^{\cc}(S)\cap \VV_{(\delta,\epsilon)}$ for all $(\delta,\epsilon)$, and hence we have
\[
D_{(\delta,\epsilon)} \subseteq \Sigma_{(\delta,\epsilon)}\subseteq \BBQ^{\cc}(S)\cap \VV_{(\delta,\epsilon)}.
\]
Since any point of $\BBQ^{\cc}(S)$ lies in some subspace of the form $\VV_{(\delta,\epsilon)}$, we must show that
\begin{equation}\label{eqn:inclusion}
\BBQ^{\cc}(S)\cap \VV_{(\delta,\epsilon)} \subseteq D_{(\delta,\epsilon)}.
\end{equation}

To complete the proof, we will identify the linear functional corresponding to each boundary facet of $\Sigma_{(\epsilon,\delta)}$, as identified in Lemma~\ref{lem:combinatorial} (note that the functional is unique, up to scalar multiple, in the vector space spanned by $\Sigma_{(\delta,\epsilon)}$); then we will show that this functional has the form $(\beta(\FF)\mapsto \zeta(\FF\cdot \cE_f))$, where $\cE_f$ is a supernatural sheaf and where $\zeta$ is one of the functionals that arises from Corollary~\ref{cor:dualconeA refined}.   We will then apply Lemma~\ref{lem:refined positivity} to see that the result is nonnegative on any $\beta(\FF)\in \BBQ^{\cc}(S)$.

For a facet of type (ii), we consider the functional:
\[
\beta(\FF)\mapsto \beta_{i,{d_i}}(\FF\cdot \cO(d_i)).
\]
If $\ee$ is a degree sequence, then by Lemma~\ref{lem:pure and supernatural}, this functional is nonzero on $\beta(\FF_{\ee})$ if and only if $e_i=d_i$.
Let $\ee$ be a degree sequence from any chain of type (ii).  Then $e_i=d_i$ if and only if $\ee=\dd$, and it thus follows that this functional corresponds to any boundary facet of type (ii).  The functional is clearly nonnegative on any $\beta(\FF)$, completing the argument for (ii).

For a facet of type (i), we consider the case where we omit the maximal element, the other case being similar.  Note that $(d_{\max})_{\delta}=0$ and $e_{\delta}<0$ for all other degree sequences $\ee$ in $P_{(\delta,\epsilon)}$.  By the same argument as in the previous paragraph, the functional 
\[
\beta(\FF)\mapsto \beta_{\delta,0}(\FF\cdot \cO)
\]
corresponds to this boundary facet and is nonnegative on any $\beta(\FF)$.

For (iii), we let $b=\codim(\dd)$.  Without loss of generality we can shift homological indices and assume that $d_j\in \ZZ$ if and only if $j\in \{0, \dots, b\}$.   We may also assume that $d_i=0$ and thus that $d_{i+1}=2$.  We fix the root sequence $f=(-d_0>-d_1>\dots >-d_{i-1}>-d_{i+2}>\dots >-d_{b})$, by which we mean $f_1=-d_0, f_2=-d_1$ and so on.  We let $\cE_f$ be any supernatural vector bundle of type $f$.  We claim that the corresponding functional is given by
\begin{equation}\label{eqn:caseiii}
\beta(\FF) \mapsto \chi_{0,0}(\FF\cdot \cE_f).
\end{equation}
We first observe that this functional is strictly positive on $\FF_{\dd}$.  Since $-d_j$ is a root of $\cE_f$ for all $j\ne i,i+1$, applying Theorem~\ref{thm:betti numbers of pairing} yields that $\FF_{\dd}\cdot \cE_f$ is a two-term complex of the form
\[
\FF_\dd\cdot \cE_f=\left[ A^N \gets A^N(-2)\right],
\]
where $N=\beta_{i,d_i}(\FF_{\dd})\cdot \gamma_{i,-d_i}(\cE_f)$.   (To see that the rank of the $A(-2)$ is also $N$, we use the following observation: since $\codim \FF_\dd = b >\dim \cE_f$, after a generic translation, we may assume by homological tranversality that $\FF_\dd\otimes \cE_f$ is exact~~\cite[Theorem, p. 335]{miller-speyer}. Then Proposition~\ref{prop:exact} implies that the alternating sum of the ranks of $\FF_\dd\cdot \cE_f$ equals $0$.) It follows that $\chi_{0,0}(\FF_{\dd}\cdot \cE_f)=N>0$.  

Next, we recall that the functional $\chi_{0,0}$ splits a Betti table on $A$ into two regions:
\[
\chi_{0,0}\Rightarrow 
\left|
\begin{array}{cccccccc}
 & \vdots& &\multicolumn{1}{|c}{\vdots}&& \vdots&&\\
\dots&0&0&\multicolumn{1}{|c}{1}&-1&1&-1&\dots\\
\dots&0&0&\multicolumn{1}{|c}{1}&-1&1&-1&\dots\\
\dots&0&0&\multicolumn{1}{|c}{1^\circ}&-1&1&-1&\dots\\ \cline{4-5}
\dots&0&0&0&0&\multicolumn{1}{|c}{1}&-1&\dots\\
\dots&0&0&0&0&\multicolumn{1}{|c}{1}&-1&\dots\\
& \vdots&&\vdots&& \multicolumn{1}{|c}{\vdots}&&\\
\end{array}
\right|,
\]
where the $1^\circ$ corresponds to $\beta_{0,0}$.   If, for some $\FF$, $\beta(\FF\cdot \cE_f)$ lies entirely in the upper region, then $\chi_{0,0}(\FF\cdot \cE_f)=\chi(\FF\cdot \cE_f)$ which equals zero whenever $\FF\cdot \cE_f$ is generically exact.  On the other hand, if $\beta(\FF\cdot \cE_f)$ lies entirely in the lower region, then $\chi_{0,0}(\FF\cdot \cE_f)$ equals the dot product of $\beta(\FF\cdot \cE_f)$ with the zero matrix.

Thus, to complete our computation for case (iii), it suffices to verify the following claim: if $\ee\leq \dd'$ then $\beta(\FF_{\ee}\cdot \cE_f)$ lies entirely in the upper region and $\FF_e\cdot \cE_f$ has the Betti table of a generically exact free complex; on the other hand, if $\ee\geq \dd''$ then $\beta(\FF_{\ee}\cdot \cE_f)$ lies entirely in the lower region.  

This claim follows from repeated applications of Lemma~\ref{lem:pure and supernatural}.  For instance, fix some $\ee\leq \dd'$.  We will verify that $\beta_{0,e_\ell}(\FF_{\ee}\cdot \cE_f)\ne 0$ only if $e_{\ell}\leq 0$ (and thus any such entry lies in the upper region determined by $\chi_{0,0}$).  By Lemma~\ref{lem:pure and supernatural}, we have
$
f_{\ell}>-e_{\ell}>f_{\ell+1}.
$
However, $e_{\ell}\leq d_{\ell}$ and thus $f_{\ell}\geq -d_{\ell}$.  By construction of $f$, this holds if and only if $\ell \leq i$.  Since $\ee$ is a degree sequence, we then have $e_{\ell}\leq e_i\leq d_i=0$, as desired.  Similar arguments verify that any nonzero entry of $\FF_{\ee}\cdot \cE_f$ automatically lies in the upper region when $\ee\leq \dd'$. Further, since $\ee\leq \dd'$ and since $\cc$ is nondecreasing, it follows that $\codim(\ee)\geq \codim(\dd')$ and thus that $\codim(\FF_{\ee})+\codim(\cE_f)\geq n+1$.  By Lemma~\ref{lem:refined positivity}, it follows that $\FF_{\ee}\cdot \cE_f$ has the Betti table of a generically exact free complex.  This completes the argument when $\ee\leq \dd'$.  By applying a similar argument in the case $\ee\geq \dd''$, we conclude that we have found the correct functional for the boundary facet of type (iii).

To complete the argument for case (iii), we must check that the functional in \eqref{eqn:caseiii}  is nonnegative on any $\beta(\FF)\in \BBQ^{\cc}(S)$.  Since the degree sequence $\dd$ is compatible with $\cc$, it follows that $c_0\leq b=\codim(\dd)\leq c_1$.  Since $\dim \cE_f=b-1$, it then follows that
$
\cE_f\in \CQ^{n+1-b}(\PP^n)\subseteq \CQ^{n+1-c_1}(\PP^n).
$
Combining Lemma~\ref{lem:refined positivity} and Corollary~\ref{cor:dualconeA refined} yields the desired nonnegativity claim.

We next consider case (iv).  Set $b=\codim(\dd)$ and shift homological indices so that $d_j\in \ZZ$ if and only if $j\in \{0, \dots, b\}$.  In this case, we will have $d_0=-\epsilon$ and $d_b=b-\delta$.  We fix the root sequence $f=(-d_{1}>-d_{2}>\dots>-d_{b-1})$ and we let $\cE_f$ be any supernatural bundle of type $f$.   The corresponding functional is:
\[
\beta(\FF)\mapsto \chi_{1,b-\delta}(\FF\cdot \cE_f).
\]
Note first that
\[
\FF_{\dd}\cdot \cE_f=\left[{A^N}(\epsilon)\gets A^N(\delta-b) \right],
\]
where $N=\beta_{0,\epsilon}(\FF_{\dd})\cdot \gamma_{0,-\epsilon}(\cE_f)$.  Hence the functional evaluates to $N>0$ on $\FF_{\dd}$.  
As in the proof of case (iii) above, the cases where we consider a degree sequence $\ee$ satisfying either $\ee\leq \dd'$  or $\ee\geq \dd''$ split in half:
if $e\leq \dd'$, then $\FF_{\ee}\cdot \cE_f$ has the Betti table of a generically exact free complex and, by Lemma~\ref{lem:pure and supernatural}, it is supported entirely in the upper region defined by $\chi_{1,-\delta-b}$; if $\ee\geq \dd''$, then $\beta(\FF_{\ee}\cdot \cE_f)$ is supported entirely in the lower region.  

We must also check the nonnegativity of this functional on any complex $\FF$ that is compatible with $\cc$.   As in case (iii), we may simply apply Lemma~\ref{lem:refined positivity} and Corollary~\ref{cor:dualconeA refined} to verify this nonnegativity.

For case (v), we let $b$ be the codimension of $\dd$.  By reindexing the columns, it suffices to consider the case where the finite entries of $\dd$ lie in homological positions $\{0,\dots,b\}$.  Note that $d_{b}=d''_b=b-\delta$ and $d''_{b+1}=b+1-\delta$.  We let $f$ be the root sequence $f=(-d_0>\dots >-d_{b-2})$ and we let $\cE_f$ be a supernatural sheaf of type $f$.  The appropriate functional has the form:
\[
\beta(\FF)\mapsto \chi_{1,b-\delta}(\FF\cdot \cE_f).
\]
By a minor adaption of the argument used for case (iii), we may confirm that this is indeed the correct functional and that is nonnegative on any $\beta(\FF)\in \BBQ^{\cc}(S)$.

We have thus verified the inclusion
\[
\BBQ^{\cc}(S)\cap V_{(\delta,\epsilon)}\subseteq D_{(\delta,\epsilon)},
\]
and this implies the theorem.
\end{proof}

\section{Monads}\label{sec:monads}

Recall that a free complex $\FF=[\dots \gets \FF_{-1}\gets \FF_0\gets \FF_1\gets \dots]$ is a \defi{free monad} for a sheaf $\cF$ on $\PP^{n}$ if 
$$
\HH_i\widetilde{\FF}\cong
\begin{cases}                    \cF&\text{ if } i=0\\
					0 &\text{ if } i\neq 0.
\end{cases}
$$					
See \cite[\S8]{eis-floy-schrey}, and the references therein, for more on free monads.  We prove a decomposition theorem for the Betti table of a free monad.
\begin{cor}\label{cor:monads}
Let $\FF$ be a free monad for a sheaf $\cF$ on $\PP^n$.  Then we may write
\[
\beta(\FF)=\lambda'\beta(\FF')+\lambda''\beta(\Hom(\FF'',S)).
\]
where: $\FF'=[\FF'_0\gets \FF'_1\gets \dots]$ is a free complex that resolves a coherent sheaf on $\PP^n$ of rank equal to the rank of $\cF$; $\FF''=[\FF''_0\gets \FF''_1\gets \dots]$ is a free complex that resolves a coherent sheaf of rank $0$ on $\PP^n$; and $\lambda'$ and $\lambda''$ are nonnegative rational scalars.
\end{cor}

\begin{proof}
Throughout this proof, we take the convention that, for any degree sequence $\dd$, $\FF_{\dd}$ denotes some pure complex of type $\dd$ and with same codimension as $\dd$.

The complex $\FF$ is compatible with the codimension sequence $\cc'=(\dots,0,0,0^{\zp},n+1,\dots)$. We may thus decompose $\beta(\FF)$ according the decomposition algorithm induced by the simplicial structure on $\BBQ^{\cc'}(S)$.  This decomposition has the form
\[
\beta(\FF)=\sum_{i=0}^s a_i\beta(\FF_{\dd^i}).
\]
where $a_i\in \QQ_{\geq 0}$ and where $\dd^0<\dd^1<\dots <\dd^s$ is a chain of degree sequences such that each $\dd^i$ is compatible with $\cc'$.

Since the $\dd^i$ form a chain, there is some maximal $r$ such that $\dd^r$ has codimension $>0$.  For $i=0\dots,r$, the entries of $\dd^{i}$ that are $>-\infty$ have strictly positive index, so the complex $\FF_{\\d^{i}}$ is concentrated in non-negative 
homological degree. Setting $D':=\sum_{i=0}^r a_i\beta(\FF_{\dd^i})$, we have 
\[
\beta_{i,j}D'=\begin{cases}
\beta_{i,j}\FF &\text{ if } i>0\\
0 & \text{ if } i<0.
\end{cases}
\]

We next consider the dual Betti table $\beta(\Hom(\FF,S))$.  Note that $\Hom(\FF,S)$ is compatible with the codimension sequence $\cc''=(\dots,0,0,0^{\zp},n+1,n+1\dots)$.  As above, we obtain a decomposition $\beta(\Hom(\FF,S))=\sum_{i=0}^{s'} b_i\beta(\FF_{\ee^i})$ according to the decomposition algorithm induced by $\BBQ^{\cc''}(S)$.  We define $D''$ to be the sum $D''=\sum_{i=0}^{t} b_i\beta(\FF_{\ee^i})$ where $t$ is the maximal value such that the degree sequence $\ee^t$ has codimension $>0$.  
As above, 
\[
\beta_{i,j}D''=\begin{cases}
\beta_{i,j}(\Hom(\FF,S))=\beta_{-i,-j}\FF &\text{ if } i>0\\
0 & \text{ if } i<0.
\end{cases}
\]

We define $(D'')^\vee\in \VV$ by the formula $\beta_{i,j}(D'')^\vee=\beta_{-i,-j}D''$, and we consider the difference of Betti tables:
\[
E:=\beta(\FF)-D'-(D'')^\vee.
\]
By construction, $E$ equals $0$, except possibly in column $0$.  

We first note that the sum of the entries of $E$ is precisely $\rank \cF$.  This follows from the fact that $D'$ and $(D'')^\vee$ are sums of rank $0$ Betti tables, and thus:
\[
\sum_j \beta_{0,j} E= \sum_{i,j} (-1)^i\beta_{i,j}(\FF)=\rank \cF.
\]

We next claim that all entries of $E$ are nonnegative.  This follows from the form of the decomposition algorithm.  Within a single column, the decomposition algorithm works from top to bottom.  So $D'$ consists of the lowest degree syzygies from column $0$.  Since dualizing inverts degrees, $(D'')^\vee$ consists of the highest degree syzygies from column $0$.  Thus, the only possible problem would be if $D'$ and $(D'')^\vee$ overlapped in column $0$, and in this case, all entries of $E$ would be strictly negative.  This would contradict the computation in the previous paragraph.

We now conclude the proof.  By construction, we may choose a scalar $\lambda'$ and a free complex $\FF'$ that is compatible with $\cc'$, such that $\lambda'\beta(\FF')=D'+E$.  Since $\FF'=[\FF'_0\gets \dots]$ and is compatible with $\cc'$, it follows that $\widetilde{\FF}'$ resolves a coherent sheaf $\cF'$ on $\PP^n$.  Further the rank of $\cF'$ equals $\sum_{j} \beta_{0,j} E$ which equals the rank of $\cF$.  Next, we may choose a scalar $\lambda''$ and a free complex $\FF''=[\FF''_0\gets \dots]$ that is compatible with $\cc''$, such that $\lambda''\beta(\FF'')=D''$.  It also follows that $\widetilde{\FF}''$ resolves a rank $0$ sheaf $\cF''$.
\end{proof}

\begin{example}[Free Monads]\label{ex:monads}
In \cite[Example 8.2]{eis-floy-schrey}, it is shown that
\[
\widetilde{\FF}:=\left[0\gets \cO_{\PP^4}(-1)^2 \gets  \cO_{\PP^4}(-2)^{11}\gets \overset{\zp}{\cO_{\PP^4}(-3)^{20}}\gets \cO_{\PP^4}(-4)^{10}\gets 0 \right]
\]
is a free monad for the ideal sheaf of a certain rational surface in $\PP^4$.  This decomposes as
\begin{align*}
\beta(\FF)&=
\begin{bmatrix}
2&11&20^\zp&10
\end{bmatrix}
\\
&=\beta(\FF')+\beta(\FF'')
\\
&=
\begin{bmatrix}
-&-&11^\zp&10
\end{bmatrix}
+
\begin{bmatrix}
2&11&9^\zp&-
\end{bmatrix}
\\
&=
\left(
\begin{bmatrix}
-&-&10^\zp&10
\end{bmatrix}
+
\begin{bmatrix}
-&-&1^\zp&-
\end{bmatrix}
\right)
+
\left( \begin{bmatrix}
2&4&2^\zp&-
\end{bmatrix}
+
\begin{bmatrix}
-&7&7^\zp&-
\end{bmatrix}
\right).
\end{align*}
\end{example}

\section{Categorified Eisenbud--Schreyer functionals}\label{sec:functionals}

We now show explicitly how the functionals $\langle -, -\rangle_{\tau,\kappa}$
introduced by Eisenbud and Schreyer in~\cite{eis-schrey1}
arise from the
categorical pairing $\Phi$. We will use the notation and definition of these
functionals from~\cite{eis-schrey-icm}. (They are the same as the ones in~\cite{eis-schrey1} but are indexed differently.)
Although the following proposition only treats
the case when $\cE$ is supernatural, we may use \cite[Theorem~0.1]{eis-schrey2} to extend this linearly to the case of an arbitrary $\cE$.

\begin{prop}\label{thm:categorified:2}
Let $\cE$ be a supernatural sheaf of dimension $s$ on $\PP^n$ with root sequence $f$.  Every Eisenbud--Schreyer functional $\langle -, \cE\rangle_{\tau,\kappa}$ may be realized as a composition of the functor $\Phi(-,\cE)$ with one of the graded partial Euler characteristic functionals from Definition~\ref{defn:chi}.

More precisely, fix any $1\leq \tau\leq s$ and $\kappa\in \ZZ$.  Then there exists $\nu\in \ZZ$ making the diagram
\[
\xymatrix{
\DD^b(S)\ar[rr]^{\langle -, \cE\rangle_{\tau,\kappa}}\ar[rd]_{\Phi(-,\cE)}&& \ZZ\\
&\DD^b(A)\ar[ru]_{\chi_{0,\nu}}&
}
\]
commute.
\end{prop}

\begin{proof} 
Set $\nu:=\min\{\max\{\kappa, -f_{\tau}-1\},-f_{\tau+1}-1\}$.  It suffices to verify the claim for one-term complexes of the form $\FF=S^1(-u)[v]$, with $u$ and $v$ arbitrary.  
If we set $f_0=\infty$ and $f_{s+1}=-\infty$, then there exists a unique $p$ such that $f_p\geq -u>f_{p+1}$, and thus  $\beta_{v-p,u}(\FF\cdot \cE)=\gamma_{p,-u}(\cE)$ while all other Betti numbers  of $\FF\cdot \cE$ are equal to $0$.  (If $-u$ is a root of $f$, then both functionals are trivially zero, and we henceforth assume that $-u$ is not a root of $f$.)

Hence $\chi_{0, \nu}(\FF\cdot \cE)$ is either $(-1)^{v-p}\gamma_{p,-u}(\cE)$ or $0$, depending on whether or not the $1$-term complex $\FF\cdot \cE$ lies in the ``upper'' or ``lower'' region of $\chi_{0,\nu}$.  Similarly, one may check from the definition in \cite{eis-schrey-icm} that $\langle \beta(\FF),\gamma(\cE)\rangle_{\tau,\kappa}$ is either $(-1)^{v-p}\gamma_{p,-u}(\cE)$ or $0$.  
It thus suffices to confirm that our two functionals are nonzero in precisely the same situations.  We first consider the case when $v-p<0$.  Then $\FF\cdot \cE$ is supported in homological degrees $<0$, and hence $\chi_{0, \nu}(\FF\cdot \cE)$ is automatically $0$. Similarly, since $v<p$ there are no terms in the definition of $\langle \beta(\FF),\gamma(\cE)\rangle_{\tau,\kappa}$ of the form $\beta_{v,u}\gamma_{p,-u}$, and hence this functional is also zero.  

It is similarly routine to check that, if $v-p>1$, then both functionals are nonzero.  We are then left with the cases $v-p=0$ and $v-p=1$.  For $v-p=0$, we have:
\begin{equation}\label{eqn:cond1}
\langle \beta(\FF),\gamma(\cE)\rangle_{\tau,\kappa}=0 \iff p> \tau \text{ or } p=\tau \text{ and } u>\kappa.
\end{equation}
We also have
\begin{equation}\label{eqn:cond2}
\chi_{0,\nu}(\FF\cdot \cE)=0 \iff u> \nu=\min\{\max\{\kappa, -f_{\tau}-1\},-f_{\tau+1}-1\}.
\end{equation}
We observe that the conditions in \eqref{eqn:cond1} and \eqref{eqn:cond2} are equivalent.  Namely, the second condition yields that either: $f_{\tau+1}\geq-u$ (which is equivalent to $p>\tau$); or $u>\kappa$ and $f_{\tau}\geq -u$ (which is equivalent to $u>\kappa$ and $p\geq \tau$). 

We now check the case $v-p=1$.  Here we have:
\begin{equation}\label{eqn:cond3}
\langle \beta(\FF),\gamma(\cE)\rangle_{\tau,\kappa}=0 \iff p> \tau \text{ or } p=\tau \text{ and } u>\kappa+1.
\end{equation}
We also have
\begin{equation}\label{eqn:cond4}
\chi_{0,\nu}(\FF\cdot \cE)=0 \iff u> \nu+1=\min\{\max\{\kappa+1, -f_{\tau}\},-f_{\tau+1}\}.
\end{equation}
Again, these conditions are equivalent: either $f_{\tau+1}>-u$ (which, given that $-u$ is not a root of $\cE$, is equivalent to $p>\tau$); or $u>\kappa+1$ and $f_{\tau}>-u$ (which, given that $-u$ is not a root of $\cE$, is equivalent to $u>\kappa+1$ and $p\geq \tau$). 
\end{proof}

\section{Revisiting the Betti table/Cohomology table duality}\label{sec:duality}


By Theorem~\ref{thm:betti numbers of pairing}, the pairing $\Phi$ induces a bilinear map:
\begin{align*}
\VV\times \WW & \to \UU\\
(\beta(\FF),\gamma(\cE))&\mapsto \beta(\FF\cdot \cE).
\end{align*}

For any $k=1, \dots, n+1$, Theorem~\ref{thm:Phi}  implies that this pairing  restricts to a map of cones:
\begin{equation}\label{eqn:3cones}
\BBQ^{k}(S)\times \CQ^{n+1-k}(\PP^n)\to \BBQ^1(A),
\end{equation}
where we recall that $\BBQ^k(S)$ is the cone of Betti tables corresponding to the codimension sequence $\cc=(\dots,k,k,\dots,)$.
Namely, since $\codim(\widetilde{\FF})+\codim(\cE)\geq n+1$, we may (after translating $\cE$ by an element of $\GL_{n+1}$) assume by~\cite[Theorem, p. 335]{miller-speyer} that $\widetilde{\FF}\otimes \cE$ is exact and thus that $\FF\cdot \cE$ is generically exact.

The 
 appearance of $\BBQ^1(A)$ explains the need for the full family of Eisenbud--Schreyer functionals used in \cite{eis-schrey1}.
\begin{proof}[Sketch of proof of Theorem~\ref{thm:duality}]
The direction (1) implies (2) follows from the map in equation \eqref{eqn:3cones}. For the other direction, we revisit the proof of Theorem~\ref{thm:extremal rays refined} in the case $\cc=(\dots, k,k,\dots)$.  In this case, all boundary facets used in the proof have one of two forms:
\[
\beta(\FF)\mapsto\beta_{i,j}(\FF\cdot \cE) \qquad \text{ or } \qquad \beta(\FF)\mapsto\chi_{i,j}(\FF\cdot \cE),
\]
where $\cE$ is a supernatural sheaf of dimension $n+1-k$.  Since $\beta_{i,j}$ and $\chi_{i,j}$ are nonnegative on $\BBQ^1(A)$, we see that hypothesis (2) implies that all boundary functionals of $\BBQ^k(S)$ are nonnegative on $v$, and hence $v$ lies in $\BBQ^k(S)$.
\end{proof}

Theorem~\ref{thm:duality}  admits a dual statement about cohomology tables.  For any $k\leq  n$, we fix a linear subspace $\PP^{k-1}\subseteq \PP^n$ and we define $\CvbQ(\PP^{k-1})$ as the subcone of $\WW$ generated by cohomology tables of vector bundles on $\PP^{k-1}$.
\begin{prop}\label{prop:other}
Fix an integer $1\leq k\leq n+1$ and a point $w$ in the vector space span of $\CvbQ(\PP^{k-1})$.   The following are equivalent:
\begin{enumerate}
	\item   $w$ is a positive, rational multiple of a cohomology table $\gamma(\cE)$ where $\cE$ is a vector bundle on $\PP^{k-1}$;
	\item  Given any free complex $\FF$ of $S$-modules of codimension $k$, $w$ pairs with $\beta(\FF)$ to give an element of $\BBQ^1(A)$.
\end{enumerate}
\end{prop}
\begin{proof}[Sketch of proof of Proposition~\ref{prop:other}]
The direction (1) implies (2) again follows from equation \eqref{eqn:3cones}.  For the other direction, we revisit \cite[Proof of Theorem~0.5]{eis-schrey1}, which illustrates that the study of $\CvbQ(\PP^{k-1})$ only requires functionals of two forms:
\[
\gamma(\cE)\mapsto\beta_{i,j}(\FF_\dd\cdot \cE) \qquad \text{ or } \qquad \gamma(\cE)\mapsto\chi_{i,j}(\FF_\dd\cdot \cE),
\]
where $\FF_\dd$ is a shifted pure resolution of codimension $k$.  Since $\beta_{i,j}$ and $\chi_{i,j}$ are nonnegative on $\BBQ^1(A)$, hypothesis (2) implies that all boundary functionals of $\CvbQ(\PP^{k-1})$ are nonnegative on $w$.
\end{proof}

\begin{example}\label{rmk:issues}
The above proposition fails if we replace $\CvbQ(\PP^{k-1})$ by $\CQ^{n+1-k}(\PP^n)$.  
Let $n=1$ and consider the point $w:=\sum_{i=0}^\infty \frac{1}{2^i} \gamma(\cO_{\PP^1}(-i))$.  
This infinite sum converges, and thus $w\in \WW$.  Further, equation \eqref{eqn:3cones} implies that the functional
\[
\beta(\FF)\mapsto \sum_{i=0}^\infty \tfrac{1}{2^i} \beta(\FF\cdot \cO_{\PP^n}(-i))
\]
induces a map $\BBQ^{2}(S)\to \BBQ^1(A)$.  However, a direct computation shows that
\[
\gamma_{1,j}(w)=\frac{1}{2^{j}}.
\]
Hence, no scalar multiple of $w$ can equal the cohomology table of a sheaf on $\PP^n$, since there is no scalar multiple of $w$ such that all of the entries of $w$ are integers.  Thus $w\notin \CQ^0(\PP^n)$.
\end{example}

\section*{\underline{{Part III: Beyond polynomial rings}}}
\section{Other $\NN$-graded rings}\label{sec:functor}
 In this section, $R$ will denote an $\NN$-graded $\kk$-algebra, not necessarily generated in degree 1, that is a finite extension of $S=\kk[x_0,\dots,x_n]$, with the variables $x_{i}$ having degree 1.  In other terms, $R$ may be any graded ring that admits a linear Noether normalization. 
 
 We will 
 show that various cones of Betti tables of
 bounded free complexex with finite length homology over $R$ are the same as the
 corresponding cones over $S$ in certain cases, generalizing Corollary~\ref{cor:isom cones}, 
 
An example of an algebra $R$ satisfying the conditions above may be constructed from any projective scheme $X$ and a finite morphism $f:X\to \PP^{n}$ by taking $R = \oplus_{i}H^{0}(f^{*}\cO(i))$. For instance, we might take
$R = \kk[s^{4}, s^{3}t, st^{3}, t^{4}]$, the homogeneous coordinate ring of the smooth rational
quartic in $\P^{3}$; or $R = \kk[x_{0}, x_{1}, y]/(y^{m}-g(x_{0}, x_{1}))$ where $x_{0},x_{1}$ have
degree 1, $y$ has degree 2, and the degree of $g$ is $2m$, which may be thought of as 
coming from an $m$-fold covering of the projective line.
 More generally, the inclusion $S\subseteq R$ induces a finite map $f\colon X=\Proj(R)\to \PP^n$. We set $L:=f^*\cO(1)$. 

We define cohomology tables and Betti tables for $R$ as follows. If $\cE\in \DD^b(X)$ then $\gamma(\cE)\in \WW$ is the table with entries $\gamma_{i,j}(\cE)=h^i(X,\cE\otimes L^{\otimes j})$.  For any $0\leq k \leq d$, we set $\CQ^{k}(X,L)$ to be the cone spanned by cohomology tables $\gamma(\cE)$ with $\codim(\cE)\geq k$.  For a bounded free complex $\FF$ of graded $R$-modules, we define the Betti table $\beta(\FF)$ by the formulas $\beta_{i,j}(\FF)=\dim \Tor_i(\FF,\kk)_j$.  
We are interested in bounded free complexes, but since $R$ will generally fail to be a regular ring, it is no longer the case that arbitrary elements of the derived category 
$\DD^b(R)$ may be represented by a bounded free complex.  
We thus restrict attention to the subcategory $\DD^{\text{perf}}(R)\subseteq \DD^b(R)$ of perfect complexes, that is, the subcategory of elements that may be
represented by a  bounded free complex.
If $\cc$ is a codimension sequence, then we define $\BBQ^{\cc}(R)$ in parallel with the definitions from \S\ref{sec:refined}.  


\begin{thm}\label{cor:new graded rings}
Let $R$ be a finite, graded extension of $S$ such that $\Proj(R)$ admits an Ulrich sheaf.  For any codimension sequence $\cc$ with $\cc_r<\infty$ for all $r$, we have  a natural isomorphism
$
\BBQ^{\cc}(S)\overset{\cong}{\to} \BBQ^{\cc}(R).
$
\end{thm}
Our assumption that $c_r<\infty$ (or equivalently that $c_r\leq n+1)$ is critical.  For instance, if $R$ is not Cohen--Macaulay and if $\FF$ is the resolution of the any finite length $S$-module, then the Auslander--Buchsbaum Theorem implies that there cannot exist any free resolution $\Gbull$ over $R$ whose Betti table is a scalar multiple of the Betti table of $\FF$.

We next define a pairing $\Phi_{R}$:
\[
\xymatrix{\DD^{\text{perf}}(R)\times \DD^{b}(X)  \ar[r]^-{\Phi_{R}}&\DD^{b}(A).}
\]
If $\FF \in \DD^{\text{perf}}(R)$ then we use the notation $\widetilde{\FF}$ to denote the corresponding complex of coherent sheaves on $X$.
Let $\sigma: R\to R\otimes A=R[t]$ be the map of rings defined by $\sigma(r)=rt^{\deg(r)}$, and
write $-\otimes_\sigma R[t]$ to denote tensoring over $R$ with $R[t]$ using the structure
given by $\sigma$. 

If $F$ is a graded  $R$-module, then 
$
F\otimes_{\sigma} S[t]
$
is a bigraded $S[t]$ module.
Thus we may define a functor $\tau$ on derived
categories that takes a graded complex of free $R$-modules $\FF$ to
\[
\tau(\FF): =\widetilde \FF \otimes_{\sigma}\cO_{X\times \AA^{1}},
\]
a complex of graded sheaves on $X\times \AA^{1}$, with the grading coming from the coordinate $t$ on $\AA^{1}$.
We let $p_2: X\times \AA^1\to \AA^1$ be the projection map, and we define $\Phi_{R}$ in parallel to the definition of $\Phi$:
\begin{defn}
Given $\FF\in \DD^{\text{perf}}(R)$ and $\cE\in \DD^b(X)$ we define
\[
\Phi_{R}(\FF,\cE):=Rp_{2*}\left( \tau(\FF)\otimes _{X\times \AA^1}\left( \cE\boxtimes \cO_{\AA^1}\right)\right).
\]
\end{defn}
As above, we often omit $\Phi_{R}$ from the notation and  write $\FF\cdot \cE$ for $\Phi_{R}(\FF,\cE)$.  We immediately obtain an analogue of Theorem~\ref{thm:Phi} for the functor $\Phi_{R}$.
\begin{theorem}\label{thm:PhiX}
If $\FF$ is a bounded complex of free graded $R$-modules and $\cE$ is a bounded complex of coherent sheaves on $X$ then:
\begin{enumerate}
	\item\label{thm:PhiX:1}  The Betti table of $\Phi_{R}(\FF,\cE)$ depends only on the Betti table of $\FF$ and the cohomology table of $\cE$.
	\item\label{thm:PhiX:2}  If $\widetilde{\FF}\otimes \cE$ is exact, then $\Phi_{R}(\FF,\cE)$ is generically exact.  
\end{enumerate}
\end{theorem}
\begin{proof}
Statement \eqref{thm:PhiX:1} follows by a minor adaptation of the proof from Theorem~\ref{thm:betti numbers of pairing} above.  Statement \eqref{thm:PhiX:2} follows by applying the proof of Proposition~\ref{prop:exact}
\end{proof}

We would like to pull back a Betti table from $S$ to $R$ in a way that is compatible with codimension sequences.  This requires the following definition.

\begin{defn}
For any $S$-algebra $R$, we say that $\FF\in \DD^b(S)$ \defi{ is homologically transverse to $R$} if we have:
\begin{itemize}
	\item $\Tor^S_q(\HH_r(\FF),R)$ has finite length for $q>0$ and for all $r$, and
	\item $\codim_S \HH_r(\FF)\otimes_S R\geq \min\{n+1, \codim_S \HH_r({\FF})+\codim_S R\}.$
\end{itemize}
\end{defn}

 \begin{lemma}\label{lem:pushpull}
Fix a codimension sequence $\cc$ with $c_r<\infty$ for all $r$, and let $\FF\in \DD^b(S)$ be compatible with $\cc$.  If $\FF$ is homologically transverse to $R$, then the left derived pullback $Lf^*(\FF)$ is compatible with $\cc$.
\end{lemma}
\begin{proof}
If we represent $\FF$ by a free complex, then we have $Lf^*\FF=f^*\FF=\FF\otimes_S R$.   We consider the spectral sequence:
\[
E_{r,q}^2=\Tor^S_{q}(\HH_r(\FF),R)\Rightarrow \HH_{q+r}(\FF\otimes_S R).
\]
Since $c_i\leq n+1$ for all $i$, and since $E^{r,q}_2$ has finite length for all $q>0$ and all $r$, it follows that
\begin{align*}
\codim_R \left(\HH_{r}(\FF\otimes_S R)\right) &\geq \min\{n+1, \codim_R\left( \HH_r(\FF)\otimes_S R\right)\} \\
&\geq \min\{n+1, \codim_S \HH_r(\FF)\}\\
& \geq c_r.
\end{align*}
Thus, $f^*\FF$ is compatible with $\cc$, as claimed.
\end{proof}

For homological transversality arguments we need an infinite field. The next result shows that
field extensions do not change the cones of rational Betti tables (we do not know whether they change the semigroups of Betti tables themselves.)

\begin{lemma}\label{lem:ground field}
Let $\cc$ be any codimension sequence for $R$ and let $K$ be a field extension of $\kk$.  Then there is a natural isomorphism $\BBQ^{\cc}(R)\cong\BBQ^{\cc}(R\otimes_{\kk}K)$.
\end{lemma}
\begin{proof}
Since $\kk\subseteq K$ is flat, we have an inclusion $\BBQ^{\cc}(R)\subseteq\BBQ^{\cc}(R\otimes_{\kk}K)$ induced by $\beta(\FF)\mapsto \beta(\FF\otimes_{\kk} K)$. Thus it suffices to
show that if $\FF$ is a bounded free complex over $R\otimes_{\kk}K$ compatible with $\cc$,
 then there is bounded free complex over $R$ that is compatible with $\cc$ and whose
 Betti table is a multiple of $\beta(\FF)$.
 
If $K$ is a finite extension of $\kk$, then since $R\otimes_{\kk}K$ is a finitely generated 
free module over $R$, the complex $\FF$, regarded as a complex of $R$-modules,
has Betti table a multiple of that of $R$.  (The multiple is $[K:\kk].$)

If $K$ is algebraic over $\kk$, then since the maps of $\FF$ can be written in terms of finitely
many matrices of finite size, $\FF$ is actually defined over a finite extension of $\kk$, so we
are reduced to the case of a finite extension above.

It now suffices to treat the case of a transcendental extension, and by the same argument as
above, we may assume that the transcendence degree is finite. Induction then reduces the
problem to the case of transcendence degree 1.

Let $\FF\in \DD^b(R\otimes_{\kk}\kk(z))$ be a free complex that is compatible with $\cc$.  Clearing denominators, we may extend $\FF$ to a complex over $\kk[z][x_0,\dots,x_n]$.  The locus of $\Spec(\kk[z])$ where the specialization of $\FF$ fails to be compatible with $\cc$ is  closed.  Since $\FF$ is compatible with $\cc$ over the generic point, it follows that there is a closed point $P\in \Spec(\kk[z])$ such that the restriction of $\FF$ to $P$ is still compatible with $\cc$.  Since the residue field of $P$ is finite over $\kk$, we are again reduced to the case of a finite extension.\end{proof}

The final ingredient for our proof of Theorem~\ref{cor:new graded rings} is the following lemma, which shows how the functors $\Phi$ and $\Phi_{R}$ interact with the map $f$.

\begin{lemma}\label{lem:general pairing}
Let $\FF\in \DD^b(S)$ be homologically transverse to $R$ and let $\cE\in \DD^b(\PP^n)$.  Then
\[
\beta(Lf^*\FF\cdot \cE)=\beta(\FF\cdot f_*\cE).
\]
\end{lemma}
\begin{proof}
We first note that $\beta(Lf^*\FF)$ equals $\beta(\FF)$.  We also have:
\[
\gamma_{i,j}(\cF)=h^i(X,\cF\otimes L^{j}) =h^i(\PP^n,f_*(\cF\otimes L^{j}))=h^i(\PP^n,(f_*\cF)\otimes \cO_{\PP^n}(j)) = \gamma_{i,j}(f_*\cF),
\]
where the first equality is by definition, the second uses that $f$ is finite so $f_*$ is exact, the third uses the projection formula and that $L^j=f^*\cO_{\PP^n}(j)$, and the fourth is by definition.
Hence $\gamma(\cE)$ equals $\gamma(f_*(\cE))$.  Theorem~\ref{thm:betti numbers of pairing} then provides a closed formula for the Betti numbers of $\Phi(\FF,f_*\cE)$ in terms of $\beta(\FF)$ and $\gamma(f_*(\cE))$.  As noted in the proof of Theorem~\ref{thm:PhiX}\eqref{thm:PhiX:1}, the same formula relates the Betti numbers of $\Phi_{R}(Lf^*\FF,\cE)$ with  $\beta(Lf^*\FF)$ and $\gamma(\cE)$.
\end{proof}

\begin{proof}[Proof of Theorem~\ref{cor:new graded rings}]
The central argument in~\cite[Proof of Theorem 5]{eis-schrey-abel} shows that the cones
$\CQ^0(X,L)$ and $\CQ^{0}(\PP^n)$ are the same;
 although $L$ was assumed to be very ample in that proof, the  argument works perfectly just assuming that $f$ is finite.

Turning to the cones of Betti tables, 
we may assume that $\kk$ is infinite by Lemma~\ref{lem:ground field}. Suppose that $\FF$ is a bounded
free complex such that $\beta(\FF)\in \BBQ^{\cc}(S)$.
By \cite[Theorem, p.\ 335]{miller-speyer}, the general translate of $\FF$ is homologically transverse to $R$,
so $\beta(f^{*}\FF) = \beta(Lf^{*}(\FF))\in \BBQ^{\cc}(R)$ by Lemma~\ref{lem:general pairing}. Thus
$Lf^{*}$ induces a map of cones $\BBQ^{\cc}(S)\to \BBQ^{\cc}(R)$.
As a map of vector spaces, $Lf^*\colon \VV\to \VV$ 
is the identity map, and thus 
$
Lf^*(\BBQ^{\cc}(S))\subseteq \BBQ^{\cc}(R).
$
It follows that $\BBQ^{\cc}(R)$ contains pure free complexes just as $Lf^*(\BBQ^{\cc}(S))$ does,

To show that  $\BBQ^{\cc}(R)$ is no larger than $Lf^*(\BBQ^{\cc}(S))$, it suffices to prove
a decomposition result analogous to Theorem~\ref{thm:extremal rays refined}. 
But the proof of Theorem~\ref{thm:extremal rays refined} in the case of $\BBQ^{\cc}(S)$ works for $\BBQ^{\cc}(R)$ as well.
Indeed, Lemma~\ref{lem:combinatorial} boils down to an analysis of degree sequences and pure Betti tables that are compatible with $\cc$; these arguments make no reference to the ring $S$, and hence they go through unchanged.  
The other key inputs for the proof of Theorem~\ref{thm:extremal rays refined} are the positivity results derived from Theorem~\ref{thm:Phi} and Lemma~\ref{lem:refined positivity}.  Theorem~\ref{thm:PhiX} plays the role of Theorem~\ref{thm:Phi}, and we can mimic the proof of Lemma~\ref{lem:refined positivity} to get an analogue for $R$-complexes that are compatible with $\cc$.
\end{proof}

\begin{example}[Curves]\label{ex:curves} Let $C$ be a projective curve, and let $L$ be any line bundle on $C$ that is generated by global sections.
Let $R$ be any homogeneous subring of $R' := \oplus_{i}H^{0}(L^{\otimes i})$ such that $R$ 
contains sections of $L$ that generate $L$. (For example, $R$ may fail to be Cohen-Macaulay or fail to be generated in degree $1$.)  We claim that $\BBQ^2(R)=\BBQ^2(\kk[x_0,x_1])$. To see this note first
that by Lemma~\ref{lem:ground field} we may assume that the ground field is infinite. It follows that there are
two sections $x_{0},x_{1}$ of degree 1 in $R$ that generate $L$, so $R$ is a finite module over $S = \kk[x_{0},x_{1}].$
By~\cite[Theorem~4.3]{eis-schrey-chow} the finite map $f: C\to \PP^{1}$ induced by $x_{0}, x_{1}$ admits
an Ulrich sheaf (one can take a general line bundle of degree genus$(C_{1})-1$ on a reduced component $C_{1}$ of $C$)
so we may apply Theorem~\ref{cor:new graded rings}.
\end{example}

\begin{example}[K3 surface in $\PP^5$]\label{ex:K3}
Let $R$ be the homogeneous coordinate ring of a K3 surface in $X\subseteq \PP^5$ that is the complete intersection of three quadrics. 
Since $X$ is a complete intersection, it admits an Ulrich bundle~\cite[Theorem~2.5]{herzog-ulrich-backelin}, and thus $\BBQ^3(R)=\BBQ^3(S)$.  It then follows from Example~\ref{ex:1441} that there cannot exist an exact sequence of the form:
\[
\cO_X\longleftarrow \begin{matrix}  \cO_X(-3)^4\\ \oplus\\ \cO_X(-4)^4\end{matrix}\longleftarrow \begin{matrix} \cO_X(-4)^4\\ \oplus\\ \cO_X(-5)^4\end{matrix} \longleftarrow \cO_X(-8)\longleftarrow 0.
\]
\end{example}

Eisenbud and Schreyer have conjectured that, when $L$ is very ample, every projective scheme admits an Ulrich sheaf~\cite[p. 543]{eis-schrey-chow}.  Corollary~\ref{cor:isom cones} provides a method for producing a counterexample.  Namely, imagine that for some $f\colon X\to \PP^n$ finite, we can produce a free complex $\FF$ of $R=\oplus_{e\in \NN} H^0(X,L^{\otimes e})$ modules such that $\beta(\FF)\in \BBQ^{n+1}(R)$ but where $\beta(\FF)$ does not decompose as a sum of shifted pure tables of codimension $n+1$.  This would imply that $\BBQ^{n+1}(R)\supsetneq \BBQ^{n+1}(S)$ and hence, by the contrapositive of Corollary~\ref{cor:isom cones}, this would imply that $X$ does not admit an Ulrich sheaf.

\section{Infinite resolutions}\label{sec:infinite}
As an application of our results from \S\ref{sec:functor}, we will show that the Betti tables of certain infinite free resolutions decompose as infinite sums of finite pure Betti tables.

\begin{cor}\label{cor:decomp infinite}
Suppose that $R$ satisfies the hypotheses of Theorem~\ref{cor:new graded rings}. If $\FF$ is a possibly infinite minimal free resolution of a finite length, graded $R$-module, then $\beta(\FF)$ decomposes as a nonnegative, rational, possibly infinite sum of the Betti tables of pure complexes with finite length homology whose degree sequences form a chain.
\end{cor}
\begin{proof}
Throughout this proof, we take the convention that, for any degree sequence $\dd$, $\mathbf{G}_{\dd}$ denotes some pure complex of type $\dd$ and with same codimension as $\dd$.  Let $n=\dim(R)-1$.  For $e> n+1$, we set
\[
\FF_{\leq e}=[\dots \gets \FF_{e-1}\gets \FF_e \gets 0 \gets \dots].
\]
Since $\FF$ resolves a finite length module, $\Hom(\FF_{\leq e},R)$ is compatible with the codimension sequence $\cc=(\dots,0,0,0^*,n+1,\dots, n+1^\zp, n+1,\dots)$, where the $0^*$ lies in homological degree $-e$.  By Theorem~\ref{cor:new graded rings} and the simplicial structure of
$\BBQ^{\cc}(R)$, we can use the greedy algorithm (as in Example~\ref{example:greedy decomposition}) to decompose 
$\beta(\Hom(\FF_{\leq e},R))$ as a positive rational linear combination of pure Betti tables whose degree sequences $\dd^0<\dots <\dd^{s_e}$ are compatible with $\cc$:
\[
\beta(\Hom(\FF_{\leq e},R))=\sum_{i=0}^{s_e} a_{i}\beta(\mathbf{G}_{\dd^i}).
\]
Since the decomposition algorithm proceeds from the right to left, the rightmost steps of the decomposition of $\beta(\Hom(\FF_{\leq e},R))$ will not depend on the value of $e$.  More specifically, if
$k$ is the minimal value such that $\dd^k_{-e}\ne -\infty$, then for all $i=0, \dots, k-1$, the value of $a_i$ will not depend on $e$.   Thus, as $e\to \infty$, the decomposition stabilizes to an infinite sum
\[
\beta(\Hom(\FF,R))=\sum_{i=0}^\infty a_i\beta(\mathbf{G}_{\dd^i}).
\]
Note that $\Hom(\mathbf{G}_{\dd^i},R)$ is a pure complex with finite length homology.  We may thus dualize to obtain the desired decomposition of $\FF$.
\end{proof}
%


\begin{example}
Let $R=\kk[x,y,z,w]/(xz,xw,yz,yw)$ and let $\FF$ be the minimal free resolution of $R/(x,y,z,w)^2$.  Then
\begin{align*}
\beta(\FF)&=\begin{bmatrix}1^\zp&-&-&-&-&\dots \\ -&6&16&38&92&\dots \end{bmatrix}\\
&=\begin{bmatrix}1^\zp&-&-&\dots \\ -&3&2&\dots \end{bmatrix}
+
\begin{bmatrix}-^\zp&-&-&-&-&\dots \\ -&3&6&3&-&\dots \end{bmatrix}
+
\begin{bmatrix}-^\zp&-&-&-&-&\dots \\ -&-&8&16&8&\dots \end{bmatrix}
+
\cdots
\end{align*}
The chain of degree sequences used in this decomposition has a maximal element, but it does not have a minimal element.
\end{example}

\section{Toric/Multigraded generalizations}\label{sec:toric}
Throughout this section $X$ will denote a projective toric variety, and $R$ will denote its
 $\Pic(X)$-graded Cox ring. We let $C$ be the $\Pic(X)$-graded ring $\kk[\NE(X)]$,
 where $\NE(X)$ is the semigroup of numerically effective divisors. 
 When $\alpha\in \NE(X)$ we write $t^{\alpha}$ for the corresponding element of $C$.

We will define a pairing $\Phi_{X}: D^{b}(R) \times D^{b}(X) \to D^{b}(C)$ that is analogous to
$\Phi$ and prove Theorem~\ref{thm:Phimulti}.  We will also define functionals $\chi_{i,\alpha}$ on $D^{b}(C)$ that are analogous to the 
$\chi_{i,j}$ and prove a positivity result for these functionals.


To define $\Phi_X$, let
$
\sigma\colon R\to R\otimes C
$
be the ring homomorphism $\sigma(f)=ft^{\deg(f)}$. 
Write $-\otimes_\sigma (R\otimes C)$ to denote the tensor product over $R$ with $R\otimes C$
 using the structure given by $\sigma$.
If $F$ is a $\Pic(X)$-graded  $R$-module, then 
\[
F\otimes_{\sigma} (R\otimes C)\]
is a $\Pic(X)\times \Pic(X)$-graded $R\otimes C$-module.
Thus we may define a functor $\tau_{X}$ on derived
categories that takes a  complex of graded free $R$-modules $\FF$ to
$$
\tau_{X}(\FF): =\widetilde \FF \otimes_{\sigma}\cO_{X\times \Spec(C)},
$$
a complex of graded sheaves on $X\times \Spec(C)$, with the grading coming from degrees in the coordinates on $\Spec(C)$. 


We now set $\Phi_{X}: \DD^{b}(R)\times \DD^b(X) \to \DD^{b}(C)$ to be
$$
\Phi_{X}(\FF,\cE) = Rp_{2*} \bigl(\tau_{X}(\FF)\otimes_{X\times\AA^{m}} (\cE\boxtimes \cO_{\Spec(C)}) \bigr)
$$
where $\FF\in \DD^{b}(R) , \cE\in \DD^b(X)$ and  $p_2: X\times \AA^{m}\to \AA^{m}$ is the projection. 

\begin{proof}[Proof of Theorem~\ref{thm:Phimulti}]
Statement \eqref{thm:Phi:1} follows by applying essentially the same proof as used for Theorem~\ref{thm:betti numbers of pairing} above.  

We now consider statement \eqref{thm:Phi:2}. Fix $\FF\in \DD^b(R)$ and $\cE\in \DD^b(X)$ such that $\widetilde{\FF}\otimes \cE$ is exact.  We claim that $\Phi_X(\FF,\cE)$ is exact over the generic point of $\Spec(C)$.  Let $Q(C)$ be the fraction field of $C$.  After tensoring by $Q(C)$, the map $\sigma$ becomes the usual inclusion $R\subset R\otimes Q(C)$ followed by the invertible change of variables $f\mapsto ft^{\deg(f)}$.  It follows that
\[
\FF':= \bigl(\tau(\FF)\otimes_{X\times\Spec(C)} (\cE\boxtimes \cO_{\Spec(C)}) \bigr)\otimes_{\cO_{\Spec(C)}} \cO_{\Spec(Q(C))}
 \cong
 \widetilde{\FF} \otimes \cE\otimes_{\cO_{\Spec(C)}} \cO_{\Spec(Q(C))}
\]
has no homology, since $\widetilde{\FF}\otimes \cE$ is exact.

Using $p_2$ to denote the restriction of $p_2$ to $X\times \Spec(Q(C))$, it follows by a spectral sequence computation that the complex $Rp_{2*}\FF'$ has no homology.  By flat base change, this is equal to the restriction of $\Phi_{X}(\FF,\cE)$ to the generic point $\Spec(Q(C))$ of $\Spec(C)$.
\end{proof}

We define $\BBirr(R)$ as the cone of multigraded Betti tables of complexes with irrelevant homology.  
We also define $\CQ^0(X)$ as the cone of multigraded cohomology tables on $X$ and $\BBQ^1(C)$ as the cone of Betti tables of free $C$-complexes
which are generically exact.
By Theorem~\ref{thm:Phimulti}, the functor $\Phi_{X}$ induces a bilinear pairing:
\begin{equation}\label{eqn:toric cones}
\BBirr(R)\times \CQ^0(X)\to \BBQ^1(C),
\end{equation}
similar to the pairing of cones in Figure~\ref{fig:bracket}.  

We next turn to the functionals $\chi_{i,\alpha}$. The  group $\Pic(X)$ admits a natural partial order where $\alpha \geq \alpha'$ whenver $\alpha-\alpha' \in \NE(X)$.  We fix a total order $\succeq$ that refines this partial order. For $i\in \ZZ$ and $\alpha\in \Pic(X)$,
we define $\chi_{i,\alpha}: \DD^b(C)\to \QQ$ to be
\[
\chi_{i,\alpha}(\FF)= \left(\sum_{\gamma\prec \alpha} \beta_{i,\gamma}(\FF) \right) +\left(\sum_{\gamma\preceq \alpha} (-1)\beta_{i+1,\gamma}(\FF)\right) + \left(\sum_{\substack{\ell > i+1\\ \gamma\in \ZZ^m}} (-1)^\ell\beta_{\ell,\gamma}(\FF) \right).
\]

We will prove:
\begin{cor}\label{cor:chiialpha}
Let $X$ be a projective toric variety. The functionals $\chi_{i,\alpha}$ are non-negative
on the cone $\BBQ^1(C)$. Thus if $\cE$ is a vector bundle on $X$ and $\FF$ is a complex of free
$\Pic(X)$-graded 
$R$-modules that has irrelevant homology, then
\[
\chi_{i,\alpha}( \Phi_X(\FF, \cE))\geq 0.
\]
\end{cor}

\begin{proof}[Proof of Corollary~\ref{cor:chiialpha}]
By Theorem~\ref{thm:Phimulti}, $\Phi_X(\FF,\cE)$ is generically exact.  So it suffices to prove that $\chi_{i,\alpha}(\Gbull)\geq 0$ whenever $\Gbull$ is generically exact.  Without loss of generality, we may assume that $i=0$ and $\alpha=0$.  We define a complex $\Gbull'$ as a projection from $\Gbull$ as follows.  If $j<0$ then $\Gbull'_j=0$; if $j>1$ then $\Gbull'_j=\Gbull_j$; if $j=0$ then $\Gbull'_0=\bigoplus_{\tau\prec 0} R(-\tau)^{\beta_{0,\tau}(\Gbull)}$; and if $j=1$ then $\Gbull'_1=\bigoplus_{\tau\preceq 0} R(-\tau)^{\beta_{1,\tau}(\Gbull)}$.  

Note that the Euler characteristic of $\Gbull'$ equals $\chi_{i,\alpha}(\Gbull)$.  So it suffices to consider where the homology of $\Gbull'$ can have positive rank.  Since $\Gbull$ is minimal, the definition of $\Gbull'$ ensures that the homology of $\Gbull'$ can only have finite rank at $\Gbull_0$ or at $\Gbull_2$, and hence the Euler characteristic will be nonnegative.
\end{proof}
\begin{example}
Let $R=\kk[x_0,x_1,y_0,y_1]$ be the Cox ring of $\PP^1\times \PP^1$, with the bigrading $\deg(x_i)=(1,0)$ and $\deg(y_i)=(0,1)$.  The irrelevant ideal of $R$ is $(x_0,x_1)\cap (y_0,y_1)$.  

Eisenbud and Schreyer have conjectured that, for some $r$, there exists a vector bundle $\cE$ on $\PP^1\times \PP^1$ of rank $9r$ with bigraded Hilbert polynomial $\chi(s_1,s_2)=r(9s_1s_2+20s_1+20s_2)$~\cite[Conjecture 2]{eis-schrey-abel}.  The existence of such a bundle would have implications for the structure of free $R$-complexes with irrelevant homology.  

For instance, this would imply that there cannot be a bigraded complex $\FF$ of $R$-modules with the following form:
\[
\FF
=\left[
R\gets 
\begin{matrix}
R^4(-2,-2)
\\
\oplus
\\
R^4(-2,-3)
\end{matrix}
\gets
\begin{matrix}
R^4(-2,-3)
\\
\oplus
\\
R^{11}(-4,-4)
\end{matrix}
\gets
\begin{matrix}
R^4(-5,-4)
\\
\oplus
\\
R^4(-4,-5)
\end{matrix}
\gets
0
\right].
\]
This follows by computing $\Phi_X(\FF, \cE)$ and then applying the same argument as in Example~\ref{ex:1441}.
\end{example}

The construction of $\Phi_X$ opens up new possibilities in the study of Betti tables over $R$ and, dually, in the study of cohomology tables on $X$.  
It is natural to ask whether
the map of cones satisfies duality properties similar to the duality discussed in \S\ref{sec:duality}.  This is an open question, even for $\PP^1\times \PP^1$.


\end{document}